\documentclass[12pt, leqno]{article}

\usepackage{amsmath}
\usepackage{amsthm}
\usepackage{amssymb}
\usepackage{latexsym}
\usepackage{graphicx}
\allowdisplaybreaks[1]
\usepackage{amsfonts}
\usepackage{ascmac}
\usepackage{color}
\usepackage{enumerate}

\theoremstyle{definition}
\theoremstyle{plain}
\allowdisplaybreaks[1]
\date{}
\renewcommand{\theequation}{\arabic{section}.\arabic{equation}}
\newtheorem{Thm}{Theorem}[section]

\newtheorem{Lemma}[Thm]{Lemma}

\newcommand{\p}{\partial}

\newcommand{\dis}{\displaystyle}

\newcommand{\N}{{\mathbb N}}
\newcommand{\R}{{\mathbb R}}

\newcommand{\ep}{\varepsilon }
\newcommand{\2}{\frac{1}{2} }

\newcommand{\tra}{{\sf T}}
\newcommand{\interface}{\Sigma}
\newcommand{\ve}{\varepsilon}

\title{\bf Mathematical analysis of \\ modified level-set equations }
\author{Dieter Bothe\footnote{Department of Mathematics, Technische Universit\"at Darmstadt, Germany. E-mail:  bothe@mma.tu-darmstadt.de},
  \, Mathis Fricke\footnote{Department of Mathematics, Technische Universit\"at Darmstadt, Germany. E-mail:   fricke@mma.tu-darmstadt.de}
 \,\,and \,\, Kohei Soga\footnote{Department of Mathematics, Faculty of Science and Technology, Keio University, Japan. E-mail:  soga@math.keio.ac.jp
}}
\begin{document}
\maketitle
\begin{abstract}
\noindent  The linear transport equation allows to advect level-set functions to represent
moving sharp interfaces in multiphase flows as zero level-sets.
A recent development in computational fluid dynamics is to modify the linear transport equation by introducing a nonlinear term to preserve certain  geometrical features of the level-set function, where the zero level-set must stay invariant under the modification. The present work establishes mathematical justification for a specific class of modified level-set equations on a bounded domain,  generated by a given smooth velocity field in the framework of the initial/boundary value problem of Hamilton-Jacobi equations. The first main result is the existence of smooth solutions defined in a time-global tubular neighborhood of the zero level-set, where an infinite iteration of the method of characteristics within a fixed small time interval is demonstrated; the smooth solution is shown to possess the desired geometrical feature. The second main result is the existence of time-global viscosity solutions defined in the whole domain, where standard Perron's method and the comparison principle are exploited. In the first and second main results, the zero level-set is shown to be identical with the original one.
The third main result is that the viscosity solution coincides with the local-in-space smooth solution in a  time-global tubular neighborhood of the zero level-set, where a new aspect of localized doubling the number of variables is utilized.

\medskip

  \medskip

\end{abstract}
\setcounter{section}{0}
\setcounter{equation}{0}
\section{Introduction}
The linear transport equation $\p_t f+v\cdot\nabla f=0$ describes the passive  advection of a scalar quantity $f$ by a velocity field $v$. We start with a brief overview of the fundamental role of the linear transport equation in fluid dynamics.
Suppose that a domain $\Omega\subset\R^3$ is occupied by a fluid.
The Lagrangian specification of a fluid flow is to look at the position of each fluid element, i.e., for each time $t\in[0,\infty)$  the position of the fluid element being at $\xi\in\Omega$ at time $\tau\in[0,\infty)$, which is called the ``fluid element $(\tau,\xi)$'',  is denoted by
$$X(t,\tau,\xi).$$
Then, the velocity of each fluid element $(\tau,\xi)$ is defined for each $t\ge0$ as
$$\frac{\p}{\p t}X(t,\tau,\xi).$$
Assuming that $X(t,\tau,\xi)=x$ is equivalent to $\xi=X(\tau,t,x)$ for all $t,\tau\in[0,\infty)$ and $x,\xi\in\Omega$, one obtains the Eulerian specification of the  fluid flow, i.e., the velocity field $v$ defined as
\begin{eqnarray*}
v:[0,\infty)\times\Omega\to\R^3,\quad v(t,x):=\frac{\p}{\p t}X(t,0,\xi)\Big|_{\xi=X(0,t,x)},
\end{eqnarray*}
which leads to
\begin{eqnarray}\label{1velocity}
\frac{\p}{\p t}X(t,0,\xi)=v(t,X(t,0,\xi)),\quad \forall\, t\ge0,\,\,\,\forall\,\xi\in\Omega.
\end{eqnarray}
Hence, $X$ can be seen as the flow of the kinematic ordinary differential equation (ODE)
\begin{eqnarray}\label{1ODE}
x'(s)=v(s,x(s)),\quad s\ge 0.
\end{eqnarray}
Let $F(t,\xi)$ be a scalar quantity at time $t$ that is associated with the fluid element $(0,\xi)$.
The Eulerian description $f$ of $F$ is defined as
$$f: [0,\infty)\times\Omega\to\R,\quad f(t,x):=F(t,\xi)\Big|_{\xi=X(0,t,x)}.$$
Suppose that $F(t,\xi)\equiv\phi^0(\xi)$,  i.e., each fluid element $(0,\xi)$ preserves the quantity $F(0,\xi)=\phi^0(\xi)$.
Then, noting that $f(t,X(t,0,\xi))\equiv F(t,\xi)$, we have the following identity
\begin{eqnarray*}
\frac{\rm d}{{\rm d} t} f(t,X(t,0,\xi))=0,\quad \forall\, t\ge0,\,\,\,\forall\,\xi\in\Omega.
\end{eqnarray*}
With \eqref{1velocity}, we find that $f$ satisfies the linear transport equation
\begin{eqnarray}\label{transport}
\left\{
\begin{array}{lll}
&\dis  \frac{\p f}{\p t}(t,x)+v(t,x)\cdot\nabla f(t,x)=0 \mbox{\quad in $(0,\infty)\times\Omega$},\medskip \\
&\dis f(0,\cdot)=\phi^0\mbox{\quad on $\Omega$},
\end{array}
\right.
\end{eqnarray}
where $\nabla=(\p_{x_1},\p_{x_2},\p_{x_3})$. It is intuitively clear (and mathematically true as well) that the solution of \eqref{transport} is given as
\begin{eqnarray}\label{1sol}
f(t,x)=\phi^0(X(0,t,x)).
\end{eqnarray}
This observation leads to the method of characteristics for more general first order PDEs.
Note that if $v$ and $\phi^0$ are not $C^1$-smooth, the meaning of solution must be generalized. A typical example of $F$ being preserved by each fluid element is the density in an incompressible fluid,  where the velocity $v$ comes from the incompressible Navier-Stokes equations. We refer to \cite{Lions} and \cite{DM} for recent development of mathematical analysis for the system of the linear transport equation and the incompressible Navier-Stokes equations. We refer also to \cite{DiPerna-Lions} and \cite{AC} for generalization of ODE-based classical theory of the linear transport equations to the case with velocity fields being less regular.

We now discuss the transport equation in the context of the level-set method in two-phase flow problems. Suppose that  $\Omega$ is occupied by two immiscible fluids (distinguished by the superscript $\pm$) in such a way that at $t=0$ the domain $\Omega$ is divided into two disjoint connected open sets and their interface: $\Omega^+(0)\subsetneq\Omega$ with $\p\Omega^+(0)\cap\p\Omega=\emptyset$ is a connected open set  filled by fluid${}^+$, $\Omega^-(0):=\Omega\setminus \overline{\Omega^+(0)}$ is  filled by fluid${}^-$ and $\interface(0):=\p\Omega^+(0)\cap\p\Omega^-(0)=\p\Omega^+(0)$ is the interface. Note that, for now, we discuss the case where $\interface(0)$ does not touch $\p\Omega$, while in Subsection 2.2 we will consider the other case.
We suppose that for each $t>0$ the open set
 \begin{eqnarray}\label{1interior}
\Omega^+(t):=X(t,0,\Omega^+(0))\qquad(\mbox{resp.  $\Omega^-(t):=X(t,0,\Omega^-(0))$})
\end{eqnarray}
is occupied by fluid${^+}$ (resp. fluid$^{-}$) and the common interface of fluid$^{\pm}$ is given as
\begin{eqnarray}\label{1interface}
\interface(t):=X(t,0, \interface(0)),
\end{eqnarray}
where the continuity of the flow $X$ implies that \eqref{1interface} is well-defined even though no unique fluid elements are associated to the points on $\interface(0)$. In other words,  the velocity field $v$ coming from the two-phase Navier-Stokes equations is assumed to be such that  \eqref{1ODE} generates a proper flow from $\bar{\Omega}$ to itself; see Section~2 below for more details.
The interface given as \eqref{1interface} is called a material interface, as opposed to a non-material interface formed by a two-phase flow with phase change (see \cite{DB} for investigations  of \eqref{1ODE} in the case of non-material interfaces).
Let $\phi^0:\bar{\Omega}\to\R$ be a smooth function such that $\phi^0>0$ on $\Omega^+(0)$ and $\phi^0<0$ on $\Omega^-(0)$, which implies that $\phi^0=0$ only on $\interface(0)$.
We assign to each fluid element $(0,\xi)$ the number $\phi^0(\xi)$. Let $F(t,\xi)$ be the label of the fluid element $(0,\xi)$ at time $t$, which must be equal to $\phi^0(\xi)$ for any $t\ge0$.
Then, the Eulerian description $f$ of $F$, i.e., $f(t,x):=F(t,\xi)|_{\xi=X(0,t,x)}=\phi^0(X(0,t,x))$, satisfies the transport equation \eqref{transport} with the representation \eqref{1sol}. In particular, we have
\begin{eqnarray*}
&&\Omega^+(t)=\{x\in\Omega\,|\, f(t,x)>0\},\quad \Omega^-(t)=\{x\in\Omega\,|\, f(t,x)<0\},\\
&&\interface(t)=\{x\in\Omega\,|\, f(t,x)=0 \},\quad \forall\, t\ge0.
\end{eqnarray*}
We call $f$ a {\it level-set function} and the linear transport equation for level-set functions the {\it level-set equation}. Throughout the paper, the level-set means the zero level of a level-set function.
Suppose that $\interface(0)$ is equal to the level-set of a $C^2$-function $\phi^0$ such that $\nabla \phi^0\neq0$ on $\interface(0)$, where $\interface(0)$ is a $C^2$-smooth closed surface (compact manifold without boundary).
If $X(t,\tau,\cdot):\Omega^+(\tau)\cup\interface(\tau)\cup\Omega^-(\tau)\to\Omega^+(t)\cup\interface(t)\cup\Omega^-(t)$ is a $C^2$-diffeomorphism for each $t,\tau\ge0$, we see that
\begin{eqnarray}\label{1non}
\nabla f(t,x)\neq0 \mbox{\quad on $\interface(t)$, $\forall\,t\ge0$},
\end{eqnarray}
and $\interface(t)$ keeps being a $C^2$-smooth closed surface for all $t>0$. In particular,  the unit normal vector $ \nu(t,x)$ and the total (twice the mean) curvature $\kappa(t,x)$ of $\interface(t)$ at each point $x$ are well-defined and represented as
\begin{eqnarray*}
 \nu(t,x)=\frac{\nabla f(t,x)}{|\nabla f(t,x)|},\quad
 \kappa(t,x)= -\nabla\cdot \nu(t,x),
\end{eqnarray*}
where $|\cdot|$ denotes the Euclidean norm. In a two-phase flow problem, the Navier-Stokes equations for the velocity field are coupled with the level-set equation on the interface through $\nu$ and $\kappa$. We refer to \cite{Abels} and \cite{pruss} for recent developments of mathematical analysis of multiphase flow problems and to \cite{Giga} for mathematical analysis of level-set methods beyond fluid dynamics.

In computational fluid dynamics, the level-set equation is often used to represent a moving interface. In this context, the level-set approach has several advantages, such as a very accurate approximation of the mean curvature and a straightforward handling of topological changes of the interface (e.g.,\ breakup and coalescence of droplets).
In a numerical simulation, it is common to choose an initial level-set function $\phi^0$ that coincides locally with the signed distance function of a given closed surface $\interface(0)$, where  $\phi^0$ is characterized by $|\nabla \phi^0| \equiv 1$ in a neighborhood of $\interface(0)$.
However, it is known that the local signed distance property is not preserved by \eqref{transport}, i.e., $f(t,\cdot)$ does not coincide even locally with the signed distance function of $\interface(t)$ for $t>0$ in general.
 In fact, a short calculation \cite{F} shows that, along each curve $x(\cdot)$ determined by \eqref{1ODE} (it is called a {\it characteristic curve}) such that $x(t)\in\interface(t)$,
  \begin{eqnarray}
 \frac{\rm d}{{\rm d} t}|\nabla f(t,x(t))| = -|\nabla f|\langle(\nabla v) \nu , \nu \rangle \, (t,x(t))
\end{eqnarray}
holds for a classical solution $f$ of the standard level-set equation \eqref{transport}.
Here, $\langle\cdot  , \cdot \rangle$ stands for the inner product of $\R^3$.
Unfortunately, problems with the numerical accuracy emerge if $|\nabla f|$  becomes too small or too large, which is the case in general, even though the non-degeneracy condition \eqref{1non} is mathematically guaranteed.
This is an important point in practice: on the one hand, it must be possible to resolve  $|\nabla f|$ by the computational mesh, which implies an upper limit for $|\nabla f|$ related to the mesh size; on the other hand, too small values of $|\nabla f|$ lead to an inaccurate positioning of the interface, the normal field, the mean curvature field, etc. in the numerical algorithm.
In order to keep the norm of the gradient approximately constant, so-called ``reinitialization'' methods \cite{Sussman1994,Sethian1996,Sussman1999} have been developed. Typically, an additional PDE is solved that computes a new function $\tilde{f}$ with the same zero contour but with a predefined gradient norm (e.g., $|\nabla \tilde{f}|=1$ on the level-set).
In \cite{SOG}, the authors developed an alternative numerical method to control the size of the gradient based on the level-set equation with a suitable source term  that is determined by an extra equation, where the reinitialization procedure was no longer necessary.
These methods might be computationally expensive. Moreover, it is known that many reinitialization methods struggle with extra difficulties if the interface touches the domain boundary $\p\Omega$ \cite{DellaRocca2014} (i.e., if a so-called ``contact line'' is formed; see Section~\ref{section:contact-line-case}).

In order to control the norm of the gradient within a single PDE, the following nonlinear modification of  the level-set equation has been introduced\footnote{This modification has first been introduced by Ilia Roisman in the lecture entitled \emph{Implicit surface method for numerical simulations of moving interfaces},
given at the international workshop on ‘Transport Processes at Fluidic Interfaces -- from Experimental to Mathematical Analysis’, Aachen, Germany, December 2011.} in the literature of computational fluid dynamics (see \cite{F, H} for details):
\begin{eqnarray}\label{m-transport}
&& \left\{
\begin{array}{lll}
&\dis  \!\!\!\!\frac{\p \phi}{\p t}(t,x)+v(t,x)\cdot\nabla \phi(t,x)=\phi(t,x)R(t,x,\nabla \phi(t,x))
\mbox{\quad in $\Theta\subseteq (0,\infty)\times\Omega$}
,\medskip \\
&\dis \!\!\!\! \phi(0,\cdot)=\phi^0
\mbox{\quad on $\Omega$}.
\end{array}
\right.
\end{eqnarray}
Note that, from here on, we rather use $\phi$ instead of $f$ to stress the fact that we deal with a modified level-set equation. Since the source term on the right-hand side is chosen proportional to the level-set function $\phi$, the modification term vanishes on the zero interface and, as seen in Section 2, one can show that the evolution of the zero level-set is unaffected by the modification (in fact, the configuration component of the characteristic ODEs for \eqref{m-transport} becomes  \eqref{1ODE} on the level-set). Moreover, a suitable choice of the nonlinear function $R$ allows to control the evolution of $|\nabla\phi|$ (at least locally at the level-set).
A formal calculation \cite{F} shows that, by choosing
\begin{align}\label{original-source-term}
R(t,x,p) := \left\langle\nabla v(t,x) \frac{p}{|p|} , \frac{p}{|p|} \right\rangle,
\end{align}
we indeed obtain
 \begin{eqnarray}\label{1modify}
 \frac{\rm d}{{\rm d} t}|\nabla \phi(t,x(t))|\equiv 0, \quad \forall\,t\ge0
\end{eqnarray}
along each characteristic curve $x(t)$ of \eqref{m-transport} such that $x(t)\in\interface(t)$ for all $t\ge0$
or, equivalently,
\begin{eqnarray*}
\dis\forall\, t\in[0, \infty),\,\,\,\forall\,x\in \interface(t),\,\,\,\exists\, \xi\in \interface(0)\mbox{ such that } |\nabla \phi(t,x)|=|\nabla\phi^0(\xi)|.
\end{eqnarray*}
We will prove this statement rigorously using the method of characteristics (see problem \eqref{problem1} in Section~\ref{section:classical_solution}). Notice that, in general, the property \eqref{1modify} only holds locally at the level-set. The signed distance function of $\interface(t)$ itself does not solve \eqref{transport} nor \eqref{m-transport} in general, but rather a non-local PDE, cf.\ Lemma~3.1 in \cite{H}. From the numerical perspective, it is of interest to study a formulation like \eqref{m-transport} because the advection of the interface and the preservation of the norm of the gradient are combined into one single PDE, i.e., into a monolithic approach. In addition to the choice \eqref{original-source-term}, we will also study a variant in which a cut-off function is applied such that the nonlinear source term is only active in a neighborhood of the level-set (see problem \eqref{RRRR} below) and another simpler modified level-set equation, which only keeps the norm of the gradient within given bounds; see the initial value problem \eqref{problem2-2} below. We refer to \cite{F} for a numerical investigation of \eqref{m-transport}.

Now, we move to the mathematical analysis of \eqref{m-transport}. It is important to note  that, due to the nonlinear source term in \eqref{m-transport}, the ODE \eqref{1ODE} is no longer the characteristic ODE of \eqref{m-transport}; instead, the system of ODEs \eqref{2chara}-\eqref{2chara3} defines the characteristic curves of \eqref{m-transport}. See Appendix 1 for more details on the method of characteristics as applied to Hamilton-Jacobi equations.
Furthermore, since \eqref{m-transport} is a first order fully nonlinear PDE, the mathematical analysis of \eqref{m-transport} is not at all as simple as that of \eqref{transport}, even if $v$ and $R$ are smooth enough.
Existence of a classical solution on the whole domain within an arbitrary time interval is no longer possible in general, i.e., the notion of viscosity solutions is necessary.  Then, it is expected that  the following statements hold true for $R$ given by \eqref{original-source-term} or its variants:
\begin{enumerate}
\item[(i)]   \eqref{m-transport} provides a level-set that is identical to the original one provided by \eqref{transport} for all $t\ge0$;
\item[(ii)]  \eqref{m-transport} admits a unique classical solution $\phi$ at least in a $t$-global tubular neighborhood of the level-set (see its definition in Section 2) so that the normal field and mean curvature field are well-defined by $\phi$ and the property \eqref{1modify} (or,  less restrictively, an a priori bound of $|\nabla\phi|$) holds on the level-set for all $t\ge0$;
\item[(iii)]  \eqref{m-transport} admits a unique  global-in-time viscosity solution defined  on $[0,\infty)\times\bar{\Omega}$;
\item[(iv)]  If initial data is $C^2$-smooth, the viscosity solution $\tilde{\phi}$ coincides  with the local-in-space classical solution $\phi$  in a $t$-global tubular neighborhood of the level-set, i.e.,  partial $C^2$-regularity of $\tilde{\phi}$.
\end{enumerate}

The purpose of the current paper is to provide full proofs of (i)-(iv) for the problem \eqref{m-transport} with a given smooth velocity field $v$ and the above-mentioned $R$, where mathematical analysis on the system of \eqref{m-transport} and Navier-Stokes type equations for $v$ is an interesting future work.
We will exploit the method of characteristics to show (ii) and (i) for the smooth solution; usually,  the method of characteristics works only within  a short time interval; however, since the nonlinearity of   \eqref{m-transport} becomes arbitrarily small near the level-set, on which $|\nabla \phi|$ is appropriately controlled as well,  we may iterate  the method of characteristics countably many times with a shrinking  neighborhood of the level-set to construct a time global solution defined  in a $t$-global tubular neighborhood of the level-set.
To show  (iii) and (i) for the viscosity solution, we will apply the standard theory of viscosity solutions to \eqref{m-transport} with a boundary condition arising formally  from the classical solutions.
To prove (iv), we adapt the idea of localized doubling the numbers of variables for the comparison principle of viscosity solutions within a cone of dependence; the difficulty is that we cannot have a cone of dependence that contains a $t$-global tubular neighborhood of the level-set; we will demonstrate a new version of localized doubling the numbers of variables with an unusual choice of a penalty function  in a $t$-global tubular neighborhood of the level-set.
We emphasize  that the result (iv) is particularly important from application points of view in the sense that, once a continuous viscosity solution is obtained, it provides the level-set, its normal field and mean curvature field with the necessary regularity being guaranteed;  numerical construction of a viscosity solution on the whole domain would be easier  than that of a local-in-space smooth solution; there is huge literature pioneered by \cite{Crandall-Lions} on rigorous numerical methods of viscosity solutions.

Finally, we compare our results on (i)--(iv) with  the work \cite{H}.
In \cite{H},  the author  formulated a modification of the initial value problem of a general  Hamilton-Jacobi equation with an autonomous Hamiltonian (including the linear transport equation with $v=v(x)$) on the whole space and proved the existence of a unique viscosity solution, where the modification  is essentially the same as \eqref{original-source-term};  owing to the modification, he showed that the (continuous) viscosity solution of the modified equation stays close to the signed distance function of its own level-set with good upper/lower estimates, from which he obtained differentiability of the viscosity solution on the level-set with the norm of the derivative to be one.  Additional regularity of the viscosity solution away from the level-set remained open. Our current paper provides a stronger partial regularity property of viscosity solutions in the same context as  \cite{H}.

\setcounter{section}{1}
\setcounter{equation}{0}
\section{$C^2$-solution on tubular neighborhood of level-set}\label{section:classical_solution}
Let $\Omega\subset\R^3$ be a bounded connected open set and  $v=v(t,x)$ be a given smooth function defined in $[0,\infty)\times\bar{\Omega}$. Our investigation relies on certain properties of the flow generated by $v$ that satisfies certain conditions; in particular, we require {\it flow invariance of $\bar{\Omega}$ and $\p\Omega$, i.e., the solution $x(s)=x(s;s_0,\xi)$ of \eqref{1ODE} with initial condition $x(s_0)=\xi$, $(s_0,\xi)\in [0,\infty)\times\bar{\Omega}$ (resp. $(s_0,\xi)\in [0,\infty)\times\p\Omega$) uniquely exists and stays in $\bar{\Omega}$ (resp. $\p\Omega$) for all  $s\in[0,\infty)$}. Note that since $\bar{\Omega}$ is compact, local-in-time invariance implies global-in-time  invariance. Typical examples of $v$ fulfilling our requirements are
\begin{itemize}
\item[(i)] $\p \Omega$ is smooth and $v(t,x)\cdot \nu(x)=0$ for all $x\in \p\Omega$ and $t\ge 0$, where $\nu$ is the unit normal of $\p\Omega$ (cf., non-penetration condition in fluid dynamics),
\item[(ii)] $v(t,x)=0$  for all $x\in \p\Omega$ and $t\ge 0$ (cf., non-slip condition in fluid dynamics).
\end{itemize}
In the current paper, we consider a more general situation based on the theory of ODEs on closed sets and flow invariance.  For this purpose, we introduce  the so-called Bouligand contingent cone $T_K(x)$ of  an arbitrary closed set  $K\subset\R^d$ at $x\in K$ as
\begin{equation*}\label{eq:10}
T_K(x):=\left\{z \in \R^d\,\Big|\, \liminf\limits_{h \to 0+} \frac{ {\rm dist}\, (x+ hz,K)}{h}=0\right\} \text{ for } x \in K.
\end{equation*}
We say that $y\in\R^d$ is subtangential to $K$ at a point $x\in K$, if $y\in T_K(x)$. Note that $T_K(x)=\R^d$ for $x$ in the interior of $K$. We shall employ the following result on  flow invariance.
\begin{Lemma}\label{flow-invariance}
Let $J=(a,b)\subset \R$, $K \subset \R^d$ compact and $g: J \times K \to \R^d$ (jointly) continuous and locally Lipschitz in $x\in K$. Then, the following holds true:
\begin{enumerate}
\item[(a)] Suppose that $\pm g$ are  subtangential to $K$, i.e.,
\begin{equation*}
\pm g(s,x) \in T_K(x), \quad\forall\, s \in J, \,\,\,\forall\,  x \in K.
\end{equation*}
Then, given any $s_0 \in J$ and $x_0 \in K$, the initial value problem
$$x'(s)=g(s,x(s)),\quad x(s_0)=x_0$$	
has a unique solution defined on $J$ that stays in $K$.
\item[(b)] Suppose that  $\pm g$ are  subtangential to $\p K$, i.e.,
$$\pm g(s,x) \in T_{\p K}(x), \quad\forall\, s \in J, \,\,\,\forall\,  x \in \p K.$$
Then, the sets $K$, $\partial K$ and $K\setminus \p K$  are flow invariant.
\end{enumerate}
\end{Lemma}
\noindent See Appendix 2 for more on flow invariance and Lemma \ref{flow-invariance}.

Now we state the hypothesis on the velocity field $v$:
\begin{itemize}
\item[(H1)]  $v\in C^0([0,\infty)\times\bar{\Omega};\R^3)\cap C^1([0,\infty)\times\Omega;\R^3)$ is locally Lipschitz in $x\in \bar{\Omega}$; $v$ is three times partially differentiable in $x$; all of the partial derivatives of $v$ belong to $C^0([0,\infty)\times \Omega;\R^3)$,
\item[(H2)]  $\pm v(s,x) \in T_{\p \Omega}(x)$ for all $s \in [0,\infty),\,\, x \in \p\Omega$,
\item[(H3)] $\displaystyle |\frac{\p v_i}{\p x_j}|$ ($i,j=1,2,3$) are bounded on $[0,\infty)\times\Omega$.
\end{itemize}
We remark that
the upcoming nonlinear modification  of the linear transport equation requires  $C^3$-smoothness of $v$ in $x$ so that its characteristic ODEs are  properly defined; due to Lemma \ref{flow-invariance}, $\bar{\Omega}$, $\p\Omega$ and $\Omega$ are flow invariant with respect to the flow $X$ of \eqref{1ODE}; $X(s,\tau,\cdot)$  is continuous on $ \bar{\Omega}$ and $C^3$-smooth in $\Omega$.
 If $\bar{\Omega}$ is a cube, for instance, (H2) implies: at each vertex,  $v$ must be equal to zero, while on each edge, $v$ may take non-zero values parallel to the edge.
\subsection{Case 1: problem with level-set being away from $\p\Omega$}
Let $\interface(0)\subset \Omega$ be a closed $C^2$-smooth surface. Let $\phi^0:\Omega\to\R$ be a $C^2$-smooth function such that
$$\mbox{$\{x\in\Omega\,|\,\phi^0(x)=0\}=\interface(0)$,\quad   $\nabla \phi^0\neq 0$ on $\interface(0)$}.$$
Let $f$ be the solution of the original level-set equation \eqref{transport}.
We keep the notation and configuration  in \eqref{1interior} and \eqref{1interface}, where we repeat
$$\interface(t):=\{x\in\Omega\,|\, f(t,x)=0\}=X(t,0,\interface(0)).$$
Note  that due to (H2) we have
$$ \interface(t)\subset \Omega,\quad \forall\,t\ge0.$$
For each $t\ge0$, let $\interface_\ep(t)$ with $\ep>0$ be the $\ep$-neighborhood of $\interface(t)$, i.e.,
$$\interface_\ep(t):=\bigcup_{x\in\interface(t)}\{y\in\R^3\,|\,|x-y|<\ep\},$$
where we  always consider $\ep>0$ such that $\interface_\ep(t)\subset \Omega$.
We say that {\it a set $\Theta\subset [0,\infty)\times\Omega$ is a ($t$-global) tubular neighborhood of the level-set $\{\interface(t)\}_{t\ge0}$, if  $\Theta$ contains
$$\bigcup_{t\ge0}\Big(\{t\}\times \interface(t)\Big),$$
and there exists a nonincreasing function $\ep:[0,\infty)\to\R_{>0}$ such that
$$\{t\}\times\interface_{\ep(t)}(t)\subset\Theta,\quad \forall\,t\ge0.$$ }
The problem under consideration is to find a  tubular neighborhood $\Theta$ of $\{\interface(t)\}_{t\ge0}$ and a $C^2$-function $\phi$ satisfying
\begin{eqnarray}\label{problem1}
&&\left\{
\begin{array}{lll}
&&\dis \frac{\p \phi}{\p t} +v\cdot \nabla \phi=\phi \Big<\big( \nabla v \big) \frac{\nabla \phi}{|\nabla \phi|}, \frac{\nabla \phi}{|\nabla \phi|}\Big>\mbox{\quad in $\Theta$},\medskip\\
&&\dis \phi(0,\cdot)=\phi^0\mbox{\quad on $\Theta|_{t=0}$}.
\end{array}
\right.
\end{eqnarray}
Note that this problem makes sense with $\phi^0$ being defined only in a neighborhood of $\interface(0)$, e.g., $\phi^0$ is given as the local signed distance function of $\interface(0)$. We state the first main result of this paper.
\begin{Thm}\label{Thm1}
There exists a  tubular neighborhood $\Theta$ of the level-set $\{\interface(t)\}_{t\ge0}$ for which \eqref{problem1} admits a unique $C^2$-solution $\phi$ satisfying
\begin{eqnarray*}
&&\dis \interface^\phi(t):= \{x\in\Omega\,|\,  \phi(t,x)=0\}=\interface(t),\quad \forall\,t\in[0,\infty),\medskip\\
&&\dis\forall\, t\in[0, \infty),\,\,\,\forall\,x\in \interface(t),\,\,\,\exists\, \xi\in \interface(0)\mbox{ such that } |\nabla \phi(t,x)|=|\nabla\phi^0(\xi)|.
\end{eqnarray*}
\end{Thm}
\begin{proof}
We treat the PDE in \eqref{problem1} as the Hamilton-Jacobi equation
\begin{eqnarray}\label{HJ22}
\p_t\phi+H(t,x,\nabla \phi,\phi)=0,
\end{eqnarray}
generated by the Hamiltonian $H: [0,\infty)\times\Omega\times\R^3\times\R\to\R$ defined as
\begin{eqnarray*}
H(t,x,p,\Phi)&:=&v(t,x)\cdot p-\Phi\Big<\nabla v(t,x) \frac{p}{|p|}, \frac{p}{|p|}\Big>\\
&=&v(t,x)\cdot p-\2\Phi\Big<\Big(\nabla v)+(\nabla v)^\tra\Big) \frac{p}{|p|}, \frac{p}{|p|}\Big>.
\end{eqnarray*}
Set
\begin{eqnarray*}
D(v):=\frac{\nabla v+(\nabla v)^\tra}{2},\quad
\tilde{D}(v,p):=\frac{\p}{\p x}\Big<\big( \nabla v(t,x) \big) \frac{p}{|p|}, \frac{p}{|p|}\Big>.
\end{eqnarray*}
The characteristic ODEs of \eqref{HJ22} are given as  (see Appendix 1 for more details)
\begin{eqnarray}\label{2chara}
 x'(s)&=& \frac{\p H}{\p p}(s,x(s),p(s),\Phi(s))\\\nonumber
 &=&v(s,x(s))\\\nonumber
 && -2\frac{\Phi(s)}{|p(s)|}\Big[
D(v(s,x(s)))\frac{p(s)}{|p(s)|}
 -\Big<D(v(s,x(s)))\frac{p(s)}{|p(s)|},\frac{p(s)}{|p(s)|}\Big> \frac{p(s)}{|p(s)|}
\Big], \\\label{2chara2}
 p'(s)&=&-\frac{\p H}{\p x} (s,x(s),p(s),\Phi(s))-\frac{\p H}{\p\Phi}(s,x(s),p(s),\Phi(s))p(s)\\\nonumber
&=&-(\nabla v(s,x(s)))^\tra p(s)
+\Big<D(v(s,x(s))) \frac{p(s)}{|p(s)|}, \frac{p(s)}{|p(s)|}\Big>p(s)\\\nonumber
&&+\Phi(s) \tilde{D}(v(s,x(s)),p(s)),\\\label{2chara3}
\ \Phi'(s)&=& \frac{\p H}{\p p}(s,x(s),p(s),\Phi(s))\cdot p(s)-H(s,x(s),p(s),\Phi(s)) \\\nonumber
&=&  \Phi(s)\Big<D( v(s,x(s))) \frac{p(s)}{|p(s)|}, \frac{p(s)}{|p(s)|}\Big> , \\\label{2chara4}
&&\!\!\!\!\!\!\!\!\!\!\!\!\!\!\!\!\!\!x(0)=\xi,\quad p(0)=\nabla\phi^0(\xi),\quad \Phi(0)=\phi^0(\xi)\quad \mbox{(except $\xi$ such that $p(0)=0$).}
\end{eqnarray}
Note that $\tilde{D}(v(t,x),p)$ is still $C^1$-smooth because of (H1), which is required for the method of characteristics.  We sometimes use the notation $x(s;\xi),p(s;\xi),\Phi(s;\xi) $ to specify the initial point.
Our proof is based on the investigation of the  variational equations of the characteristic ODEs \eqref{2chara}--\eqref{2chara3} for each $\xi \in\interface(0)$ to ensure the invertibility of $x(s;\cdot)$ in a small neighborhood of $\interface(s)$ for each $s\ge0$. In a general argument of the method of characteristics, such invertibility is proven only within a small time interval. Below, we will show an iterative scheme to extend the time interval in which the  invertibility holds with a shrinking neighborhood of $\interface(s)$ as $s$ becomes larger.

For each $\xi \in\interface(0)$, as long as  $x(s;\xi),p(s;\xi),\Phi(s;\xi) $ exist, it holds that
\begin{eqnarray}\label{1xsxs}
\Phi(s;\xi)&\equiv&\Phi(0;\xi)=0,\\\label{2xsxs}
x'(s) &=&v(s,x(s)), \quad x(s)\in\interface(s),\\\label{3xsxs}
p'(s)&=&-(\nabla v(s,x(s)))^\tra p(s)
+\Big<D(v(s,x(s))) \frac{p(s)}{|p(s)|}, \frac{p(s)}{|p(s)|}\Big>p(s),\\\label{4xsxs}
p'(s)\cdot p(s)&=&\2\frac{\rm d}{{\rm d} s}|p(s)|^2=0,\quad |p(s)|^2\equiv|p(0)|^2=|\nabla \phi^0(\xi)|^2\neq0,
\end{eqnarray}
where we note that due to (H2) applied to \eqref{2xsxs}, the above equalities \eqref{1xsxs}--\eqref{4xsxs}  hold for all $s\ge0$ together with
$$0<\inf_{\interface(0)}|\nabla \phi^0|\le |p(s)|\le \sup_{\interface(0)}|\nabla \phi^0|<\infty,\quad\forall\,s\ge0.$$
Hence,  for each $\xi \in\interface(0)$, there exist $\frac{\p x}{\p\xi}(s)=\frac{\p x}{\p\xi}(s;\xi)$ and $\frac{\p \Phi}{\p\xi}(s)=\frac{\p \Phi}{\p\xi}(s;\xi)$ for all $s\ge0$ satisfying
\begin{eqnarray}\label{1vvv2}
&&\frac{\rm d}{{\rm d} s}\frac{\p x}{\p\xi}(s)=\nabla v(s,x(s))\frac{\p x}{\p\xi}(s)-\frac{2}{|p(s)|}  b(s)\otimes  \frac{\p \Phi}{\p \xi}(s),\quad \frac{\p x}{\p\xi}(0)=id,\\\label{2vvv2}
&&\frac{\rm d}{{\rm d} s}\frac{\p\Phi}{\p\xi}(s)=\frac{\p\Phi}{\p\xi}(s)\Big<D( v(s,x(s))) \frac{p(s)}{|p(s)|}, \frac{p(s)}{|p(s)|}\Big> , \quad \frac{\p\Phi}{\p\xi}(0)=\nabla \phi^0(\xi),
\end{eqnarray}
where
$$b(s):=D(v(s,x(s)))\frac{p(s)}{|p(s)|}
 -\Big<D(v(s,x(s)))\frac{p(s)}{|p(s)|},\frac{p(s)}{|p(s)|}\Big> \frac{p(s)}{|p(s)|},\quad |p(s)|\equiv|p(0)|.$$
Hereafter, $V_0,V_1,V_2,V_3,V_4$ will denote constants depending only on $v$, $\sup_{\interface(0)}|\nabla\phi^0|<\infty$ and $\inf_{\interface(0)}|\nabla\phi^0|>0$.
Due to  (H3),
it holds that for any $\xi\in\interface(0)$,
\begin{eqnarray*}
&&|\nabla v(s,x(s)) y|\le V_0|y|,\quad \forall\,y\in\R^3,\\
&&\Big| D(v(s,x(s)))\frac{p(s)}{|p(s)|} \Big|\le V_1,\quad
\Big| \Big<D(v(s,x(s)))\frac{p(s)}{|p(s)|},\frac{p(s)}{|p(s)|}\Big> \frac{p(s)}{|p(s)|} \Big|\le V_1,\quad |b(s)|\le V_1.
\end{eqnarray*}
Observe that for each $\xi \in\interface(0)$, we have from \eqref{2vvv2},
\begin{eqnarray*}
&&\frac{\rm d}{{\rm d} s}\frac{\partial\Phi}{\p \xi}(s)=\frac{\partial\Phi}{\p \xi}(s)\Big<D( v(s,x(s))) \frac{p(s)}{|p(s)|}, \frac{p(s)}{|p(s)|}\Big>,\\
&&\Big|\frac{\p\Phi}{\p\xi}(s)\Big|=|\nabla \phi^0(\xi) e^{\int_0^s \langle D(v(s,x(\tau))) p(\tau), p(\tau)\rangle|d\tau }|\le  V_2e^{V_1s }, \quad \forall\,s\ge0,\\
&&\Big|\frac{\p \Phi}{\p \xi_i}(s) b(s)\Big|\le V_1V_2e^{V_1s }\quad (i=1,2,3),\quad \forall\,s\ge0.
\end{eqnarray*}
For each $i=1,2,3$ and $\xi \in\interface(0)$, we have from \eqref{1vvv2}
\begin{eqnarray*}
\frac{\rm d}{{\rm d} s}\frac{\p x}{\p\xi_i}(s)&=&\nabla v(s,x(s))\frac{\p x}{\p\xi_i}(s)-\frac{2}{|p(0)|}\frac{\p \Phi}{\p \xi_i}(s) b(s),\quad \frac{\p x}{\p\xi_i}(0)=e^i,\\
\frac{\rm d}{{\rm d} s}\frac{\p x}{\p\xi_i}(s)\cdot \frac{\p x}{\p\xi_i}(s)&=&\nabla v(s,x(s))\frac{\p x}{\p\xi_i}(s)\cdot \frac{\p x}{\p\xi_i}(s)-\frac{2}{|p(0)|}\frac{\p \Phi}{\p \xi_i}(s) b(s)\cdot \frac{\p x}{\p\xi_i}(s),\\
\2\frac{\rm d}{{\rm d} s}\Big|\frac{\p x}{\p\xi_i}(s)\Big|^2
&\le& V_1\Big|\frac{\p x}{\p\xi_i}(s)\Big|^2+ \frac{2}{\dis \inf_{\interface(0)}|\nabla\phi^0|} V_1V_2e^{V_1s}\Big|\frac{\p x}{\p\xi_i}(s)\Big|\\
&\le&  (V_1+  1) \Big|\frac{\p x}{\p\xi_i}(s)\Big|^2+  \Big(\frac{V_1V_2e^{V_1s}}{\dis \inf_{\interface(0)}|\nabla\phi^0|}\Big)^2.
\end{eqnarray*}
Gronwall's inequality implies
\begin{eqnarray}\label{est21}
\Big|\frac{\p x}{\p\xi_i}(s;\xi)\Big|^2&\le& e^{2(1+V_1)s}\Big\{1+V_1\Big(\frac{V_2}{\dis \inf_{\interface(0)}|\nabla\phi^0|}\Big)^2(e^{2V_1s} -1)\Big\}\\\nonumber
&\le& V^2_3,\quad \forall\, s\in[0,1],\,\,\,\forall\,\xi\in\interface(0).\end{eqnarray}
It follows from \eqref{est21} that
\begin{eqnarray}\label{est22}
&\dis \Big|\frac{\rm d}{{\rm d} s}\frac{\p x}{\p\xi_i}(s;\xi)\Big|\le V_0 \Big|\frac{\p x}{\p\xi_i}(s;\xi)\Big|+2V_1V_2 e^{V_1s}\le V_4, \quad \forall\, s\in[0,1],\,\,\,\forall\,\xi\in\interface(0).
\end{eqnarray}
Due to the continuity in \eqref{2chara}--\eqref{2chara4} and  \eqref{est21}--\eqref{est22}, we find $\ep_1>0$ such that
\begin{eqnarray}\nonumber
&& \interface_{\ep_1}(0)\subset \Omega,\quad \mbox{ \eqref{2chara}-\eqref{2chara4} is solvable  for  all $s\in[0,1]$ and $\xi\in \interface_{\ep_1}(0)$},\\\label{est23}
&&
\Big|\frac{\rm d}{{\rm d} s}\frac{\p x}{\p\xi_i}(s;\xi)\Big|< 2V_4, \quad \forall\, s\in[0,1],\,\,\,\forall\,\xi\in\interface_{\ep_1}(0).
\end{eqnarray}
Fix a number $t_\ast\in(0,1]$ such that
\begin{eqnarray}\label{t-star}
1- \{3\cdot (2V_4 t_\ast)+6\cdot (2V_4 t_\ast)^2+6\cdot(2V_4 t_\ast)^3\}>0.
\end{eqnarray}
We will show that
\begin{eqnarray}\label{2contra}
&\mbox{\it \quad  $\exists\, \tilde{\ep}_1\in(0,\ep_1]$ such that  $\interface_{\tilde{\ep}_1}(0)\ni\xi\mapsto x(s;\xi)$ is injective for each $0\le s\le t_\ast$}.
\end{eqnarray}
First, we give an auxiliary explanation on how to prove \eqref{2contra}. In order to prove the injectivity of $x(s;\cdot):\Sigma_{\ep}(0)\to x(s;\Sigma_{\ep}(0))$ for each fixed $s\in[0,s_\ast]$ ($\ep>0$ and $s_\ast>0$ are some constants), we would take one of the following strategies:
\begin{itemize}
\item[(i)]  Fix $\ep>0$ and take a sufficiently small $s_\ast>0$,
\item[(ii)] Fix $s_\ast>0$ and take a sufficiently small $\ep>0$,
\end{itemize}
where we note that $\det D_\xi x(s;\xi)\neq0$ everywhere on $\Sigma_{\ep}(0)$ is not enough for the injectivity in general.
In our proof of Theorem \ref{Thm1}, we need to repeat the argument from $[0,s_\ast]$ within $[s_\ast,2s_\ast]$ to demonstrate infinite iteration with the fixed $s_\ast>0$. Hence, we take the strategy (ii) with $s_\ast=t_\ast$ given in \eqref{t-star} and sufficiently small $\ep=\tilde{\ep}_1\in (0,\ep_1]$.
Then, what we need to show is that
$$ x(s;\tilde{\zeta})=x(s;\zeta)\mbox{ for $\zeta,\tilde{\zeta}\in \Sigma_{\ep}(0)$}\Rightarrow \tilde{\zeta}=\zeta.$$
If the line segment joining $\zeta$ and $\tilde{\zeta}$ is included in $\Sigma_{\ep}(0)$, we may immediately apply Taylor's approximation to have
$$0=x(s;\tilde{\zeta})-x(s;\zeta)=\Lambda (\tilde{\zeta}-\zeta) \quad \mbox{with some $(3\times3)$-matrix $\Lambda$}$$
and get $\tilde{\zeta}=\zeta$ via $\det \Lambda>0$, where  $t_\ast$ in \eqref{t-star} is given so that $\det \Lambda>0$ holds in $\Sigma_{\ep_1}(0)$.
However, it is not a priori clear if the line segment joining $\zeta$ and $\tilde{\zeta}$ is included in $\Sigma_{\ep}(0)$ for each $s\in[0,t^\ast]$.
This is a major difficulty to prove \eqref{2contra} directly.

Now, we  prove \eqref{2contra} by contradiction.
Suppose that \eqref{2contra} does not hold. Then, we find a sequence $\{\delta_j\}_{j\in\N}\subset(0,\ep_1]$  with $\delta_j\to0$ as $j\to\infty$ and $\{t_j\}_{j\in\N}\subset(0,t_\ast]$ for which
\begin{eqnarray}\label{22contra}
\exists\,\zeta,\tilde{\zeta}\in \interface_{\delta_j}(0)\mbox{ such that }\zeta\neq\tilde{\zeta},\,\,\,x(t_j;\zeta)=x(t_j;\tilde{\zeta})\mbox{ for each $j\in\N$}.
\end{eqnarray}
Define $d_j :=\sup\{ |\tilde{\zeta}-\zeta|\,|\,  \eqref{22contra}\mbox{ holds}  \}$. We prove that $d_j\to0$ as $j\to\infty$. If not, we find $\eta>0$ and a subsequence of $\{d_j\}_{j\in\N}$ (still denoted by the same symbol) such that $d_j\ge \eta$ for all $j$.
We further take out subsequences so that $t_j\to s\in[0,t_\ast]$ as $j\to\infty$.
Then, for each $j\in\N$, there exist $\zeta_j$ and $\tilde{\zeta}_j$ such that
$$|\zeta_j-\tilde{\zeta}_j|\ge \frac{\eta}{2},\quad {\rm dist}(\zeta_j,\interface(0))\le \delta_j,\quad
{\rm dist}(\tilde{\zeta}_j,\interface(0))\le \delta_j.
$$
Taking subsequences if necessary, we see that  $\zeta_j$, $\tilde{\zeta}_j$ converge to some $\xi,\tilde{\xi}\in \interface(0)$ as $j\to\infty$, respectively, where $|\xi-\tilde{\xi}|\ge \frac{\eta}{2}$. The limit $j\to\infty$ in $x(t_j;\zeta_j)=x(t_j;\tilde{\zeta}_j)$ implies $x(s;\xi)=x(s;\tilde{\xi})$ with $\xi,\tilde{\xi}\in\interface(0)$ and $\xi\neq\tilde{\xi}$.
This is a contradiction, because $x(s;\cdot)|_{\Sigma(0)}$ is given by the flow of \eqref{1ODE}.
Consequently, $d_j\to0$ as $j\to\infty$.
Hence, for all $j$ sufficiently large,  we find  that $\zeta,\tilde{\zeta}$ in \eqref{22contra},  denoted by $\zeta_j$, $\tilde{\zeta}_j$, are included in the $\ep_1$-neighborhood of some $\xi_j\in\interface(0)$;
 for $i=1,2,3$,  Taylor's approximation within the $\ep_1$-neighborhood of $\xi_j$ yields
\begin{eqnarray*}
0&=&x_i(t_j;\tilde{\zeta}_j)-x_i(t_j,\zeta_j)\\
&=&x_i(0;\tilde{\zeta}_j)-x_i(0;\zeta_j)+(x_i(t_j;\tilde{\zeta}_j)- x_i(t_j,\zeta_j))-(x_i(0;\tilde{\zeta}_j)-x_i(0;\zeta_j))\\
&=&(\tilde{\zeta}_j-\zeta_j)_i+\frac{dx_i}{ds}(\lambda_{ji}t_j;\tilde{\zeta}_j)t_j- \frac{dx_i}{ds}(\lambda_{ji}t_j,\zeta_j))t_j, \quad \lambda_{ji}\in (0,1),\\
&=&(\tilde{\zeta}_j-\zeta_j)_i+\frac{\rm d}{{\rm d} s}\frac{\p x_i}{\p\xi}(\lambda_{ji}t_j;\tilde{\zeta}_j +\lambda'_{ji}(\tilde{\zeta}_j -\zeta_j ) )t_j\cdot (\tilde{\zeta}_j-\zeta_j), \quad \lambda'_{ji}\in (0,1).
\end{eqnarray*}
Therefore, by \eqref{est23}, we obtain
\begin{eqnarray*}
&&(I+A)(\tilde{\zeta}_j-\zeta_j)=0,\\
&&\mbox{ $I$ is the  ($3\times3$)-identity matrix,  $A$ is a ($3\times3$)-matrix with $|A_{kl}|<2V_4 t_j\le 2V_4 t_\ast$},\\
&&\det(I+A)\ge1- \{3\cdot (2V_4 t_\ast)+6\cdot (2V_4 t_\ast)^2+6\cdot(2V_4 t_\ast)^3\}>0,
\end{eqnarray*}
which leads to $\tilde{\zeta}_j=\zeta_j$. This is a contradiction and we conclude \eqref{2contra}.
Note that we also obtain $\det \frac{\p x}{\p \xi}(s;\xi)>0$ for all $s\in[0,t_\ast]$ and $\xi\in\interface_{\tilde{\ep}_1}(0)$; this follows from $ \frac{\p x}{\p \xi}(s;\xi)= \frac{\p x}{\p \xi}(0;\xi)+ \frac{\p x}{\p \xi}(s;\xi)- \frac{\p x}{\p \xi}(0;\xi)=I+[\frac{\rm d}{{\rm d} s}  \frac{\p x_k}{\p \xi_l}(\lambda_{kl}s;\xi)s]$, $\lambda_{kl}\in(0,1)$.

Define
\begin{eqnarray*}
&& O_{\tilde{\ep}_1}(t_\ast) :=\bigcup_{0\le s\le t_\ast}\Big( \{s\}\times \{x(s;\xi)\,|\,\xi\in\interface_{\tilde{\ep}_1}(0)\}\Big),\\
&&\psi_1: [0,t_\ast]\times \interface_{\tilde{\ep}_1}(0)\to  O_{\tilde{\ep}_1}(t_\ast),\quad \psi_1(s,\xi):=(s, x(s;\xi)),\\
&&\psi_1(t,\cdot):  \interface_{\tilde{\ep}_1}(0)\to  O_{\tilde{\ep}_1}(t_\ast)|_{s=t}=\{t\}\times \{x(t;\xi)\,|\,\xi\in\interface_{\tilde{\ep}_1}(0)\},\quad 0\le t\le t_\ast.
\end{eqnarray*}
The invertibility of $\xi\mapsto x(s;\xi)$ and $\det \frac{\p x}{\p \xi}(s;\xi)>0$ imply that $\psi_1$ and $\psi_1(t,\cdot)$
are $C^1$-diffeomorphic and that there exists $\ep_2>0$ such that
\begin{eqnarray}\label{ep2}
 \{x(t_\ast;\xi)\,|\,\xi\in\interface_{\tilde{\ep}_1}(0)\}\supset
\interface_{\ep_2}(t_\ast).
\end{eqnarray}
Let $\varphi_1:  O_{\tilde{\ep}_1}(t_\ast)\to \interface_{\tilde{\ep}_1}(0)$ be defined as
$$\psi_1(t,\xi)=(t,x)\Leftrightarrow (t,\xi)=\psi^{-1}_1(t,x)=:(t,\varphi_1(t,x)),$$
where we note that $\varphi_1(t,\interface(t))=\interface(0)$.
 Then, it follows from the general results of the method of characteristics (see Appendix 1) that the function
$$\phi_1: O_{\tilde{\ep}_1}(t_\ast) \to\R,\quad \phi_1(t,x):=\Phi(t; \varphi_1(t,x))$$
is $C^2$-smooth and satisfies
\begin{eqnarray*}
&&\frac{\p \phi_1}{\p t} +v\cdot \nabla \phi_1=\phi_1 \Big<\big( \nabla v \big)\frac{\nabla \phi_1}{|\nabla \phi_1|}, \frac{\nabla \phi_1}{|\nabla \phi_1|}\Big>,\quad (t,x)\in  O_{\tilde{\ep}_1}(t_\ast),\\
&&\phi_1(0,x)=\phi^0(x),\quad x\in  \interface_{\tilde{\ep}_1}(0),\\
&&\{x\,|\,\phi_1(s,x)=0\}=\interface(s),\quad \forall\,s\in[0,t_\ast],\\
&&  \nabla \phi_1(s,x)=p(s;\varphi_1(s,x)),\quad  \forall\,(s,x)\in  O_{\tilde{\ep}_1}(t_\ast),\\
&&|\nabla \phi_1(s,x)|=|p(s;\varphi_1(s,x))|=|\nabla\phi^0(\varphi_1(s,x))| ,\quad \forall\, s\in[0,t_\ast],\,\,\, \forall\, x\in \interface(s).
\end{eqnarray*}
\indent The argument up to now has used the $C^2$-smoothness of $\phi^0$ on a neighborhood of $\interface(0)$ and the upper/lower bound of $|\nabla \phi^0|$ on $\interface(0)$, where we note that $|\nabla \phi_1(t_\ast,\cdot)|$ on $\interface(t_\ast)$ has exactly the same upper/lower bound as $|\nabla \phi^0|$ on $\interface(0)$.
Therefore,  {\it we may replace $\phi^0$ with $\phi_1(t_\ast,\cdot)$ to demonstrate the same kind of estimates in terms of the above constants $V_0,V_1,V_2,V_3,V_4$ and $t_\ast$ as well as $\ep_2 >0$ appropriately chosen in \eqref{ep2} for the characteristic ODEs \eqref{2chara}-\eqref{2chara3}  for $s\in [t_\ast,2t_\ast]$ with the initial condition
$$x(t_\ast)=\xi,\quad p(t_\ast)=\nabla\phi_1(t_\ast,\xi),\quad \Phi(t_\ast)= \phi_1(t_\ast,\xi),\quad \xi\in \{x(t_\ast;\tilde{\xi})\,|\,\tilde{\xi}\in\interface_{\tilde{\ep}_1}(0)\}.$$}
Furthermore, we find a constant $\tilde{\ep}_2\in(0,\ep_2]$ such that
\begin{eqnarray*}
&&O_{\tilde{\ep}_2}(2t_\ast) :=\bigcup_{t_\ast\le s\le 2t_\ast} \Big(\{s\}\times \{x(s;\xi)\,|\,\xi\in\interface_{\tilde{\ep}_2}(t_\ast)\}\Big),\\
&&\psi_2: [t_\ast,2t_\ast]\times \interface_{\tilde{\ep}_2}(t_\ast)\to  O_{\tilde{\ep}_2}(2t_\ast),\quad \psi_2(s,\xi):=(s, x(s;\xi)),\\
&&\psi_2(t,\cdot):  \interface_{\tilde{\ep}_2}(t_\ast)\to  O_{\tilde{\ep}_2}(2t_\ast)|_{s=t}=\{t\}\times \{x(t;\xi)\,|\,\xi\in\interface_{\tilde{\ep}_2}(t_\ast)\},\quad t_\ast\le t\le 2t_\ast,
\end{eqnarray*}
are $C^1$-diffeomorphic. Also, there exists $\ep_3>0$ such that
$$ \{x(2t_\ast;\xi)\,|\,\xi\in\interface_{\tilde{\ep}_2}(t_\ast)\}\supset \interface_{\ep_3}(2t_\ast).$$
Let $\varphi_2: O_{\tilde{\ep}_2}(2t_\ast)\to \interface_{\tilde{\ep}_2}(t_\ast)$ be defined as
$$\psi_2(t,\xi)=(t,x)\Leftrightarrow (t,\xi)=\psi_2^{-1}(t,x)=:(t,\varphi_2(t,x)),$$
where we note that $\varphi_2(t,\interface(t))=\interface(t_\ast)$.
 Then, the function
$$\phi_2: O_{\tilde{\ep}_2}(2t_\ast) \to\R,\quad \phi_2(t,x):=\Phi(t; \varphi_2(t,x))$$
is $C^2$-smooth and satisfies
\begin{eqnarray*}
&&\frac{\p \phi_2}{\p t} +v\cdot \nabla \phi_2=\phi_2 \Big< \big( \nabla v \big) \frac{\nabla \phi_2}{|\nabla \phi_2|}, \frac{\nabla \phi_2}{|\nabla \phi_2|}\Big>,\quad (t,x)\in  O_{\tilde{\ep}_2}(2t_\ast),\\
&&\phi_2(0,x)=\phi_1(t_\ast, x),\quad x\in  \interface_{\tilde{\ep}_2}(t_\ast),\\
&&\{x\,|\,\phi_2(s,x)=0\}=\interface(s),\quad \forall\,s\in[t_\ast,2t_\ast],\\
&&  \nabla \phi_2(s,x)=p(s;\varphi_2(s,x)),\quad  \forall\,(s,x)\in  O_{\tilde{\ep}_2}(2t_\ast),\\
&&|\nabla \phi_2(s,x)|=|p(s;\varphi_2(s,x))|=|\nabla\phi_1(t_\ast, \varphi_2(s,x))|\\
&&\qquad =|\nabla\phi^0(\varphi_1(t_\ast,\varphi_2(s,x)))| ,\quad \forall\, s\in[t_\ast,2t_\ast],\,\,\, \forall\, x\in \interface(s).
\end{eqnarray*}
\indent Note that $\phi_1$ and $\phi_2$ are smoothly connected at $t=t_\ast$. With the common constant $t_\ast$, we may repeat this process with $\ep_1,\tilde{\ep}_1,  \ep_2,\tilde{\ep}_2, ,\ep_3, \tilde{\ep}_3,  \cdots$ (note that  it is possible that $\ep_k,\tilde{\ep}_k,\to0$ as $k\to \infty$).  We conclude the proof with defining $\Theta:= \cup_{l\in\N}  O_{\tilde{\ep}_l}(lt_\ast)$.
\end{proof}

If we choose $\phi^0$ which coincides  with the (local) signed distance function of $\interface(0)$, then the solution $\phi$ obtained in Theorem \ref{Thm1} satisfies
$$|\nabla \phi(t,x)|\equiv1 \quad\forall\, t\ge0,\,\,\,\forall\,x\in\interface(t).$$
\subsection{Case 2: problem with level-set touching $\p\Omega$}\label{section:contact-line-case}

We consider the problem \eqref{problem1} with the level-set touching the boundary of $\Omega$.
Let $K\subset \R^3$ be a bounded connected open set such that $K\cap \Omega\neq\emptyset$, $\p K\cap \p\Omega\neq\emptyset$ and $\p K$ is a closed $C^2$-smooth surface. Define
$$\interface(0):=\p K \cap \bar{\Omega}.$$
Let $\phi^0$ be a $C^2$-smooth $\R$-valued function defined in an open set containing  $\bar{K}\cup\bar{\Omega}$ such that
$$\phi^0>0\mbox{ in $\bar{K}$},\quad \phi^0<0\mbox{ outside $ \bar{K}$},\quad \nabla \phi^0\neq0 \mbox{ on $\p K$}.$$
Then, we have
$$\phi^0>0\quad\mbox{in $K\cap\Omega$},\quad\{x\in\bar{\Omega}\,|\,\phi^0(x)=0\}=\interface(0),\quad \nabla \phi^0\neq0 \mbox{ on $\interface(0)$}.$$
Let $f$ be the solution of the original level-set equation \eqref{transport}.
Define
$$\interface(t):=\{x\in\bar{\Omega}\,|\, f(t,x)=0\}=X(t,0,\interface(0)),\quad \forall\,t\ge0,$$
where we note that $\p\Omega$ is invariant under the flow $X$ of \eqref{1ODE}, and hence $\interface(t)$ always touches $\p\Omega$.
If we follow the same argument as given in Subsection 2.1, we would face non-trivial issues at/near $\interface(t)\cap\p\Omega$ coming from the behavior of the variational equations of the characteristic ODEs on $\p\Omega$.  Hence, we modify the reasoning of Subsection 2.1 so that $\interface(t)\cap\p\Omega$ is not involved.

Let $\{\Omega^k\}_{k\in\N}$ be a monotone approximation of $\Omega$, i.e., each $\Omega^k$ is  an open subset of $\Omega$; $\Omega^k\subset \Omega^{k+1}$ for all $k\in\N$; for any $G\subset \Omega$ compact, there exists $k=k(G)$ such that $G\subset\Omega^k$. Introduce
$$\{ \interface^k(0) \}_{k\in\N}, \quad  \interface^k(0):=\interface(0)\cap\Omega^k,\quad \interface^k(t):=X(t,0,\interface^k(0)),$$
where for each $k\in\N$ we have $\ep>0$ depending on $k$ such that
$$\interface^k_{\ep}(0):=\bigcup_{x\in \interface^k(0)}\{ y\in\R^3\,|\, |y-x|<\ep \}\subset \Omega.$$
Now, we may follow the same argument as given in Subsection 2.1 with $\interface^k(0)$ in place of $\interface(0)$ to obtain the following objects:
\begin{eqnarray*}
&&O_{\tilde{\ep}_1(k)}(t_\ast(k)) =\bigcup_{0\le s\le t_\ast(k)} \Big(\{s\}\times \{x(s;\xi)\,|\,\xi\in\interface^k_{\tilde{\ep}_1(k)}(0)\}\Big),\\
&&O_{\tilde{\ep}_2(k)}(2t_\ast(k)) =\bigcup_{t_\ast(k)\le s\le 2t_\ast(k)} \Big(\{s\}\times \{x(s;\xi)\,|\,\xi\in\interface^k_{\tilde{\ep}_2(k)}(t_\ast(k))\}\Big),\ldots,\\
&&\Theta^k:=  \bigcup_{l\in\N}  O_{\tilde{\ep}_l(k)}(lt_\ast(k)),\\
&&\mbox{a unique $C^2$-solution $\phi^k$ of \eqref{problem1}}|_{\Theta=\Theta^k}.
\end{eqnarray*}
The method of characteristics implies that $\phi^k\equiv\phi^{k'}$ on $\Theta^k\cap\Theta^{k'}$ for every $k,k'\in\N$.
Therefore, setting
\begin{eqnarray}\label{tube22}
\Theta:=\bigcup_{k\in\N} \Theta^k,
\end{eqnarray}
 we obtain a unique $C^2$-solution $\phi$ of \eqref{problem1}. We remark that
 $$\Theta\cap(\{t\}\times\interface(t))=\{t\}\times(\interface(t)\setminus\p\Omega),\quad\forall\,t\ge0.$$
We summarize the result:
\begin{Thm}\label{Thm122}
There exists a  tubular neighborhood $\Theta$ in the sense of  \eqref{tube22} of the level-set $\{\interface(t)\}_{t\ge0}$ touching $\p\Omega$   for which \eqref{problem1} admits a unique $C^2$-solution $\phi$ satisfying
\begin{eqnarray*}
&&\dis \interface^\phi(t):= \{x\in\Omega\,|\,  \phi(t,x)=0\}=\interface(t)\setminus \p\Omega,\quad \forall\,t\in[0,\infty),\medskip\\
&&\dis\forall\, t\in[0, \infty),\,\,\,\forall\,x\in \interface(t)\setminus \p\Omega,\,\,\,\exists\, \xi\in \interface(0)\setminus \p\Omega\mbox{ such that } |\nabla \phi(t,x)|=|\nabla\phi^0(\xi)|.
\end{eqnarray*}
\end{Thm}

Let us note in passing that another method to investigate a problem with the level-set touching $\p\Omega$ would run via smooth extension of $v$ outside $\Omega$. The following steps would suffice:
\begin{itemize}
\item[]{\bf Step~1.} Extend the velocity field $v$ to $\R\times\R^3$ as a $C^3$-function by means of Whitney's extension theorem \cite{W} or the extension operators in Sobolev spaces (see, e.g., Chapter 5 of \cite{Adams}) together with the Sobolev embedding theorem, where additional conditions on $v$ and $\Omega$ are required accordingly.
\item[] {\bf Step~2.}   Extend the flow $X(t,\tau,\xi)$ to $t,\tau\in\R$, $\xi\in\R^3$ and the  problem \eqref{problem1} to a tubular neighborhood $\tilde{\Theta}$ of the level-set $\{ X(t,0,\p K) \}_{t\ge0}$.
\item[]{\bf Step~3.} Solve the extended problem in the same way as Subsection 2.1, where each characteristic curve starting at a point of $\Omega$ stays inside $\Omega$ forever due to the flow  invariance of $\bar{\Omega}$ and $\p\Omega$ under $X$.
\end{itemize}
\noindent The restriction of the solution obtained in Step 3 to $\Theta:=\tilde{\Theta}\cap ([0,\infty)\times\bar{\Omega})$ then is the desired object.

\subsection{Simpler  nonlinear modification}
The nonlinear modification in \eqref{problem1} is designed to preserve $|\nabla \phi|$ along each characteristic curve on the level-set.
If we relax the requirement, i.e., we only ask for an a priori bound of $|\nabla \phi|$ on the level-set, we may use a much simpler modification.

We take the same configuration of the level-set considered in Subsection 2.1.
Let $\alpha>0$ be a constant such that
$$-\alpha|p^2|\le \langle\!(\nabla v(t,x))p,p\rangle\le \alpha|p|^2,\quad \forall\,p\in\R^3.$$
With a constant $\beta>\alpha$, we consider
\begin{eqnarray*}
\frac{\p\phi}{\p t}+v\cdot\nabla\phi=\phi(\beta-|\nabla\phi|),\quad \phi(0,\cdot)=\phi^0,
\end{eqnarray*}
which is seen as the Hamilton-Jacobi equation \eqref{HJ22} with
$$H(t,x,p,\Phi):=v\cdot p-\Phi(\beta-|p|).$$
The characteristic ODEs in this case are given as
\begin{eqnarray*}
x'(s)&=&\frac{\p H}{\p p}(s,x(s),p(s),\Phi(s))=v(s,x(s))+\Phi(s)\frac{p(s)}{|p(s)|},\\
p'(s)&=&-\frac{\p H}{\p x}(s,x(s),p(s),\Phi(s))-\frac{\p H}{\p\Phi}(s,x(s),p(s),\Phi(s))p(s)\\
&=&(\nabla v(s,x(s)))^{\sf T} p(s)-(|p(s)|-\beta)p(s),\\
\Phi'(s)&=&\frac{\p H}{\p p}(s,x(s),p(s),\Phi(s))\cdot p(s)-H(s,x(s),p(s),\Phi(s))=\Phi(s),\\
&&\!\!\!\!\!\!\!\!\!\!\!\!\!\!\!\! \!\!\!\!\!\!\!\!\!\!\!\! x(0)=\xi,\quad p(0)=\nabla\phi^0(\xi),\quad \Phi(0)=\phi^0(\xi)\quad \mbox{(except $\xi$ such that $p(0)=0$).}
\end{eqnarray*}
As long as the above characteristic ODEs have  solutions, $|p(s)|^2$ evolves as
$$\2\frac{\rm d}{{\rm d} s}|p(s)|^2=\langle\!(\nabla v(s,x(s)))p(s),p(s)\rangle-(|p(s)|-\beta)|p(s)|^2,$$
which leads to
\begin{eqnarray*}
\frac{\rm d}{{\rm d} s}|p(s)|^2\le -2(|p(s)|-\beta-\alpha)|p(s)|^2, \quad
\frac{\rm d}{{\rm d} s}|p(s)|^2\ge -2(|p(s)|-\beta+\alpha)|p(s)|^2.
\end{eqnarray*}
If $\beta-\alpha\le |\nabla \phi^0|\le \beta+\alpha$ on $\interface(0)$, we have for each $\xi\in\interface(0)$,
$$\beta-\alpha\le |p(s)|\le \beta+\alpha\quad\mbox{as long as $p(s)$ exists}.$$
In fact, suppose that there exists $\tau>0$ such that $|p(\tau)|>\beta+\alpha$;  set $s^\ast:=\sup\{ s\le \tau\,|\, |p(s)|=\beta+\alpha \}$; the continuity of $|p(\cdot)|$ implies that  $|p(s^\ast)|=\beta+\alpha$ and $|p(s)|> \beta+\alpha$ for all $s\in(s^\ast,\tau]$;  then, we necessarily have
$$|p(\tau)|^2\le |p(s^\ast)|^2-2\int^\tau_{s^\ast} (|p(s)|-\beta-\alpha)|p(s)|^2ds <|p(s^\ast)|^2,$$
which is a contradiction; a fully analogous argument yields the lower bound.
Therefore, we may apply the reasoning of Subsection 2.1 to the following problems: find a  tubular neighborhood $\Theta$ of $\{\Gamma(t)\}_{t\ge0}$ and a $C^2$-function $\phi$  satisfying
\begin{eqnarray}\label{problem2-2}
&&\left\{
\begin{array}{lll}
&&\dis \frac{\p \phi}{\p t} +v\cdot \nabla \phi=\phi (\beta -|\nabla\phi|)\mbox{\quad in $\Theta$},\medskip\\
&&\dis \phi(0,\cdot)=\phi^0\mbox{\quad on $\Theta|_{t=0}$}.
\end{array}
\right.
\end{eqnarray}
We obtain the following result.
\begin{Thm}\label{Thm1-1}
Suppose that  $\beta-\alpha\le |\nabla \phi^0|\le \beta+\alpha$ on $\interface(0)$. Then, there exists a  tubular neighborhood $\Theta$ of the level-set $\{\interface(t)\}_{t\ge0}$ for which \eqref{problem2-2} admits a  unique $C^2$-solution $\phi$ satisfying
\begin{eqnarray*}
&&\interface^\phi(t):= \{x\,|\,  \phi(t,x)=0\}=\interface(t),\quad \forall\,t\in[0,\infty),\medskip\\
&&\beta-\alpha \le |\nabla \phi(t,x)|\le \beta+\alpha,\quad \forall\, t\in[0, \infty),\,\,\,\forall\,x\in \interface(t).
\end{eqnarray*}
\end{Thm}
Note that Theorem~\ref{Thm1-1} is interesting for numerical purposes because it is usually not important to exactly keep $|\nabla \phi| \equiv 1$, but rather to stay away from extreme values of the gradient norm. Therefore, the problem \eqref{problem2-2} should be investigated in more detail from the numerical perspective in the future.
%
%
%
\setcounter{section}{2}
\setcounter{equation}{0}
\section{Viscosity solution on the whole domain}\label{section:viscosity-solution}

Let $\Omega\subset\R^3$ be a bounded connected open set.
Let $v=v(t,x)$ be a given function belonging to $C^0([0,\infty)\times\bar{\Omega};\R^3)\cap C^1([0,\infty)\times\Omega;\R^3)$ and satisfying   (H2)--(H3) (we do not need $C^3$-regularity in $x$).
Let $\interface(0)\subset \Omega$ be a closed $2$-dimensional surface (topological manifold) and let $\phi^0:\bar{\Omega}\to\R$ be a $C^0$-function such that $\{x\in\Omega\,|\,\phi^0(x)=0\}=\interface(0)$, where $f(t,x):=\phi^0(X(0,t,x))$ is not necessarily a classical solution of the original linear transport equation.
We keep the notation and configuration  in \eqref{1interior} and \eqref{1interface} with the current $X$ and $\phi^0$, where we repeat
$$\interface(t):=\{x\in\Omega\,|\, f(t,x)=0\}=X(t,0,\interface(0)),\quad \interface(t)\cap\p\Omega=\emptyset,\quad \forall\,t\ge0.$$
%
\indent If we deal with the Hamilton-Jacobi equation in  \eqref{problem1} on the whole space $\Omega$, the nonlinear term would suffer from singularity, i.e., $\langle(\nabla v) p , p \rangle |p|^{-2} $ cannot be continuous at $p=0$. Hence, also for a simpler structure near $\p\Omega$, we introduce a smooth cut-off so that the nonlinearity is smooth and effective only around the level-set $\{\interface(t)\}_{t\ge0}$, where we note carefully that  the level-set $\{\interface(t)\}_{t\ge0}$ is not the one determined by the upcoming viscosity solution, i.e., such cut-off can be given  independently from the unknown function.  We also note that $\nabla v$  is not required on $\p\Omega$ due to the cut-off.  We explain two ways to introduce a suitable cut-off.

The first way is simpler but available only for the problem within each finite time interval.
Fix an arbitrary terminal time $T>0$.
For $\ep>0$, let $K_{\ep} \subset \bar{\Omega}$ denote the intersection of $\bar{\Omega}$ and the $\ep$-neighborhood of $\p\Omega$.  Then, since $\interface(t)$ never touches $\p\Omega$ within $[0,T]$,  we find $\ep>0$ such that
$$K_{3\ep} \cap \interface(t)=\emptyset\quad(\mbox{or, equivalently, } K_{3\ep}\subset \Omega^-(t) ),\quad \forall\, t \in[0,T].$$
Define the  continuous function
$R_T:[0,T]\times\bar{\Omega}\times\R^3\to\R$  as
\begin{eqnarray*}
&&R_T(t,x,p) :=\eta_0(x)\eta_2(|p|)\Big<\big( \nabla v(t,x)\big) \frac{p}{|p|},\frac{p}{|p|}\Big> , \\
&&\eta_0(x):=\left\{
\begin{array}{cll}\medskip
1&&\mbox{ for $x\not\in K_{3\ep}$}, \\\nonumber \medskip
0&&\mbox{ for  $x\in K_{2\ep}$},
\\\nonumber \medskip
\mbox{non-negative smooth transition from $1$ to $0$}&&\mbox{ otherwise},
\end{array}
\right.\\
&&\eta_2(r):=\left\{
\begin{array}{cll}\medskip
1&&\mbox{ for $\dis  \frac{2}{3}\le r\le \frac{4}{3} $},
\\\nonumber\medskip
0&&\mbox{ for $\dis r\le \frac{1}{3}$ or $\dis \frac{5}{3}\le r$}, \\\medskip
\mbox{monotone smooth transition from $1$ to $0$}&&\mbox{ otherwise}.
\end{array}
\right.
\end{eqnarray*}
This choice of cut-off is uncomplicated since we do not need detailed information on asymptotics of dist$(\interface(t),\p\Omega)$ as $t\to\infty$; however, $R_T$ depends on the terminal time and a time global analysis is impossible. Note that the above specific choice of $1/3,2/3$, etc., in $\eta_2$ is not essential.

The second way is to allow $t$-dependency for cut-off.
 Since $\interface(t)$ never touches $\p\Omega$ within $[0,\infty)$,  we find a smooth function $\ep:[0,\infty)\to\R_{>0}$ (possibly  $\ep(t)\to0$ as $t\to\infty$) such that
$$K_{3\ep(t)} \cap \interface(t)=\emptyset\quad(\mbox{or, equivalently, } K_{3\ep(t)}\subset \Omega^-(t) ),\quad \forall\, t \in[0,\infty).$$
Then, we take a smooth function $\eta_1:[0,\infty)\times\bar{\Omega}\to[0,1]$  such that
\begin{eqnarray*}
&&\eta_1(t,x)=\left\{
\begin{array}{cll}\medskip
1&&\mbox{ for $x\not\in K_{3\ep(t)}$}, \\\nonumber \medskip
0&&\mbox{ for  $x\in K_{2\ep(t)}$},
\\\nonumber \medskip
\mbox{non-negative smooth transition from $1$ to $0$}&&\mbox{ otherwise},
\end{array}
\right.
\end{eqnarray*}
where we omit an explicit formula of such $\eta_1$,  and  define the  continuous function $R$ as
\begin{eqnarray}\label{RRRR}
&&R(t,x,p) :=\eta_1(t,x)\eta_2(|p|)\Big<\big( \nabla v(t,x)\big) \frac{p}{|p|},\frac{p}{|p|}\Big>,
\end{eqnarray}
where $\eta_2$ is the one in $R_T$.

Due to (H3), there exists a constant $V_0>0$ such that
\begin{eqnarray}\label{bbb}
\sup|R_T|\le V_0,\quad \sup|R|\le V_0.
\end{eqnarray}
\indent In the rest of the paper, we take $R$ given as in \eqref{RRRR}. Note that all upcoming results hold also for $R_T$ as long as the terminal time $T$ is unchanged. For an arbitrary $T>0$, we discuss existence of a unique viscosity solution $\phi$ of
\begin{eqnarray}\label{HJ}
&& \left\{
\begin{array}{rll}
\dis\!\!\! \frac{\p \phi}{\p t}(t,x)+v(t,x)\cdot\nabla \phi (t,x)\!\!\!&=&\!\!\! \phi (t,x)R(t,x,\nabla \phi(t,x)) \mbox{ in $(0,T)\times\Omega$},\\
\dis \phi(0,x)\!\!\!&=&\!\!\!\phi^0(x) \mbox{ on $\Omega$},\\[0.5ex]
\dis \phi(t,x)\!\!\!&=&\!\!\!\phi^0(X(0,t,x)) \mbox{ on $[0,T]\times\p\Omega$}.
\end{array}
\right.
\end{eqnarray}
Before going on with \eqref{HJ},  we recall the definition of viscosity (sub/super)solutions of a general first order Hamilton-Jacobi equation of the form
\begin{eqnarray}\label{HJ2}
G(z,u(z),\nabla_z u(z))=0 \mbox{\quad in $O$},
\end{eqnarray}
where $O\subset\R^N$ is an open set, $G=G(z,u,q):O\times\R\times\R^N\to\R$ is a given continuous function and $u:O\to\R$ is the unknown function. Our evolutional  Hamilton-Jacobi equation is also seen in this form with $z=(t,x)$.  In the literature, \eqref{HJ2} is often treated as a typical example of degenerate second order PDEs (i.e., the second order term is completely degenerate to be $0$).
To state the definition, we introduce the upper semicontinuous envelope $u^\ast:O\to\R$ and the lower semicontinuous envelope $u_\ast:O\to\R$ of a locally bounded function $u:O\to\R$ as
\begin{eqnarray*}
&&u^\ast(z):=\lim_{r\to0}\,\sup\{ u(y)\,|\,y\in O,\,\,\,0\le|y-z|\le r\},\\
&&u_\ast(z):=\lim_{r\to0}\,\inf\{ u(y)\,|\,y\in O,\,\,\,0\le|y-z|\le r\}.
\end{eqnarray*}
Note that $u^\ast$ is upper semicontinuous and $u_\ast$ is lower semicontinuous;  if $u$ is upper semicontinuous (resp. lower semicontinuous), we have $u=u^\ast$ (resp. $u=u_\ast$).
 \medskip

\noindent {\bf Definition.} {\it  A function $u:O\to\R$ is a viscosity subsolution (resp. supersolution) of \eqref{HJ2}, provided
\begin{itemize}
\item $u^\ast$ is bounded from the above (resp. $u_\ast$ is bounded from the below);
\item If $(\varphi,z)\in C^1(O;\R)\times O$ satisfies
$$\max_{y\in O} (u^\ast(y)-\varphi(y))=u^\ast(z)-\varphi(z)
\quad(\mbox{resp. }\min_{y\in O} (u_\ast(y)-\varphi(y))=u_\ast(z)-\varphi(z))
,$$
 we have
 $$G(z,u^\ast(z),\nabla_z \varphi(z))\le0
 \quad(\mbox{resp. } G(z,u_\ast(z),\nabla_z \varphi(z))\ge0)
 .$$
\end{itemize}
A function $u:O\to\R$ is a viscosity solution of \eqref{HJ2}, if it is both a  viscosity subsolution and supersolution of  \eqref{HJ2}.
}
\medskip

\noindent It is well-known that $\max,\min$ in the definition can be replaced by the local maximum, local minimum, respectively.

\subsection{Existence of viscosity solution}

We state the main result of this subsection.
\begin{Thm}\label{Thm2}
Let $T>0$ be arbitrary.
There exists a viscosity solution $\phi\in C^0([0,T)\times\bar{\Omega};\R)$ of \eqref{HJ}, i.e., $\phi$ satisfies the first equation in \eqref{HJ} in the sense of the above definition and the initial/boundary condition strictly. Such a viscosity solution is unique.
Furthermore, it holds that
$$\interface^\phi(t):=\{x\in\Omega \,|\,\phi(t,x)=0 \}=\interface(t),\quad \forall\,t\in[0,T).$$
\end{Thm}
\medskip

\noindent{\bf Remark.} {\it Theorem \ref{Thm2} implies  the existence of a unique viscosity solution $\phi\in C^0([0,\infty)\times\bar{\Omega};\R)$ of \eqref{HJ} with $\interface^\phi(t)=\interface(t)$ for all $t\in[0,\infty)$. }
\medskip

\begin{proof}[{\it Proof of Theorem \ref{Thm2}.}]

Introduce the function $S:[0,T]\times\bar{\Omega}\to\R$ as
\begin{eqnarray*}
&&S(t,x):=\left\{
\begin{array}{cll}
&-V_0\mbox{\qquad for $ x\in\overline{\Omega^+(t)}$, $t\in[0,T]$},\\
&V_0\mbox{\,\,\,\,\,\qquad for $ x\in\Omega^-(t)\setminus K_{\varepsilon(T)}$, $t\in[0,T]$},\\
&0\mbox{\qquad \,\,\quad for $ x\in K_{\varepsilon(T)/2}$, $t\in[0,T]$},\\
&\mbox{non-negative smooth transition from $V_0$ to $0$}\mbox{\, otherwise},
\end{array}
\right.
\end{eqnarray*}
where $\ep(t)\ge \ep (T)$ for all $t\in[0,T]$ and  $S$ is  such that $S\equiv0$ near $\p\Omega$ for all $t\in[0,T]$. Note that $S$ is smooth except on $\cup_{0\le t\le T}(\{t\}\times\interface(t))$.

{\bf Step 1: viscosity subsolution/supersolution.} We claim that the function $\rho:[0,T)\times\bar{\Omega}\to\R$ defined as
$$\rho(t,x):=\phi^0(X(0,t,x))e^{\int_0^t S(s,X(s,t,x))ds}\quad (\mbox{$X$ is the flow of \eqref{1ODE}})$$
is a viscosity subsolution of \eqref{HJ} satisfying the initial/boundary condition strictly.
For this purpose, we mollify $\phi^0$ by the standard Friedrichs' mollifier in $\R^3$ with the parameter $\delta>0$, the result being denoted by $\phi_\delta^0$.
We consider the function
$$\rho_\delta(t,x):=\phi^0_\delta(X(0,t,x))e^{\int_0^t S(s,X(s,t,x))ds},$$
where $\rho_\delta\to\rho$ locally uniformly on $[0,T)\times\Omega$ as $\delta\to0+$.
For each point $(t,x)\in(0,T)\times \Omega\setminus \cup_{0\le s<T}(\{s\}\times\interface(s))$, there exists a sufficiently small neighborhood of $(t,x)$ on which  $\rho_\delta$ is $C^1$-smooth and satisfies
\begin{eqnarray*}
 \frac{\p \rho_\delta}{\p t}(t,x)+v(t,x)\cdot\nabla \rho_\delta (t,x)&=& \rho_\delta(t,x) S(t,x)
\end{eqnarray*}
for all sufficiently small $\delta>0$.

{\bf  Case 1:} fix any $(t,x)\in (0,T)\times\Omega \setminus \cup_{0\le s<T}(\{s\}\times\interface(s))$ such that $x\in \Omega^+(t)$.
Then, $\rho(t,x)>0$, and hence $\rho_\delta(t,x)>0$ for all sufficiently small $\delta>0$.
Let $\varphi$ be any test function satisfying the condition in the definition of  viscosity subsolution at $(t,x)$, i.e., $\rho-\varphi$ attains the maximum at $(t,x)$ within the closed $r$-ball $\overline{B_r(t,x)}$ with the center $(t,x)$ and some $r>0$.
Let $(t_\delta,x_\delta)$ be the maximum point of $\rho_\delta-\varphi$ within  $\overline{B_r(t,x)}$, where $(t_\delta,x_\delta)\to (t,x)$ as $\delta\to0+$.
Since $\rho_\delta$ is $C^1$-smooth near $(t_\delta,x_\delta)$,  we see that $\varphi (s,y)+(\rho_\delta(t_\delta,x_\delta)-\varphi(t_\delta,x_\delta))$  is tangent to $\rho_\delta$ at $(s,y)=(t_\delta,x_\delta)$ from the above. This implies that for all sufficiently small $\delta>0$,
\begin{eqnarray*}
&&\frac{\p \varphi}{\p t}(t_\delta,x_\delta)+v(t_\delta,x_\delta)\cdot\nabla \varphi(t_\delta,x_\delta) = \frac{\p \rho_\delta}{\p t}(t_\delta,x_\delta)+v(t_\delta,x_\delta)\cdot\nabla \rho_\delta (t_\delta,x_\delta)\\
&&\quad = \rho_\delta(t_\delta,x_\delta)S(t_\delta,x_\delta) =-\rho_\delta(t_\delta,x_\delta)V_0
\le \rho_\delta(t_\delta,x_\delta)R(t_\delta,x_\delta,\nabla \varphi(t_\delta,x_\delta)).
\end{eqnarray*}
Sending $\delta\to0+$, we obtain
\begin{eqnarray*}
&&\frac{\p \varphi}{\p t}(t,x)+v(t,x)\cdot\nabla \varphi(t,x) = \frac{\p \rho}{\p t}(t,x)+v(t,x)\cdot\nabla \rho(t,x) \\
&&\quad = \rho(t,x)S(t,x) =-\rho(t,x)V_0
\le \rho(t,x)R(t,x,\nabla \varphi(t,x)).
\end{eqnarray*}
\indent  {\bf Case 2:} fix any $(t,x)\in (0,T)\times\Omega \setminus \cup_{0\le s<T}(\{s\}\times\interface(s))$ such that $x\in \Omega^-(t)$.  Then, $\rho(t,x)<0$, and hence $\rho_\delta(t,x)<0$ for all sufficiently small $\delta>0$.
 Let  $\varphi$ be any test function satisfying the condition in the definition of  viscosity subsolution at $(t,x)$, i.e., $\rho-\varphi$ attains the maximum at $(t,x)$ within the ball $\overline{B_r(t,x)}$ with some $r>0$.
Let $(t_\delta,x_\delta)$ be the maximum point of $\rho_\delta-\varphi$ within  $\overline{B_r(t,x)}$, where $(t_\delta,x_\delta)\to (t,x)$ as $\delta\to0+$.
Since $\rho_\delta$ is $C^1$-smooth near $(t_\delta,x_\delta)$,  we see that $\varphi (s,y)+(\rho_\delta(t_\delta,x_\delta)-\varphi(t_\delta,x_\delta))$  is tangent to $\rho_\delta$ at $(s,y)=(t_\delta,x_\delta)$ from the above. This implies that for all sufficiently small $\delta>0$,
\begin{eqnarray*}
&&\frac{\p \varphi}{\p t}(t_\delta,x_\delta)+v(t_\delta,x_\delta)\cdot\nabla \varphi(t_\delta,x_\delta) = \frac{\p \rho_\delta}{\p t}(t_\delta,x_\delta)+v(t_\delta,x_\delta)\cdot\nabla \rho_\delta (t_\delta,x_\delta)\\
&&\quad = \rho_\delta(t_\delta,x_\delta)S(t_\delta,x_\delta) =\rho_\delta(t_\delta,x_\delta)V_0\le \rho_\delta(t_\delta,x_\delta)R(t_\delta,x_\delta,\nabla \varphi(t_\delta,x_\delta)).
\end{eqnarray*}
Sending $\delta\to0+$, we obtain
\begin{eqnarray*}
&& \frac{\p \varphi}{\p t}(t,x)+v(t,x)\cdot\nabla \varphi(t,x) = \frac{\p \rho}{\p t}(t,x)+v(t,x)\cdot\nabla \rho(t,x) \\
&&\quad =\rho(t,x)S(t,x)
\le \rho(t,x)R(t,x,\rho(t,x),\nabla \varphi(t,x)).
\end{eqnarray*}

\indent {\bf Case 3:} fix any $(t,x)\in \cup_{0< s<T}(\{s\}\times\interface(s))$.
Then, $\rho(t,x)=0$.
Let  $\varphi$ be any test function satisfying the condition in the definition of   viscosity subsolution at $(t,x)$. We  see that
\begin{eqnarray}\label{superdiff}
&&\limsup_{(s,y)\to(t,x)}\frac{\rho(s,y)-\rho(t,x)-(\varphi(s,y)-\varphi(t,x))}{|(s,y)-(t,x)|}\\\nonumber
&&=\limsup_{(s,y)\to(t,x)}\frac{\rho(s,y)-\rho(t,x)-D\varphi(t,x)\cdot((s,y)-(t,x))}{|(s,y)-(t,x)|}\le 0,
\end{eqnarray}
 i.e., $D\varphi(t,x)=(\frac{\p\varphi}{\p t}(t,x),\nabla\varphi(t,x))\in D^+\rho(t,x)$, where $D^+\rho(t,x)$ stands for the superdifferential\footnote{
 $\dis D^+\rho(z):=\Big\{a\in\R^4\,\Big|\, \limsup_{\tilde{z}\to z}\frac{\rho(\tilde{z})-\rho(z)-  a\cdot(\tilde{z}-z) }{|\tilde{z}-z|}\le0\Big\}$ with $z=(t,x)$.} of $\rho$ at $(t,x)$.
 Since $\rho(s,X(s,t,x))=\rho(t,x)= 0$ for all $0< s\le t$ and
 $$\lim_{s\to t-0}\frac{X(s,t,x)-x}{s-t}=\lim_{s\to t-0}\frac{X(s,t,x)-X(t,t,x)}{s-t}=v(t,x),$$
 \eqref{superdiff} with $y=X(s,t,x)$ and $s\to t-0$ implies that
 \begin{eqnarray*}
 D\varphi(t,x)\cdot (1,v(t,x))\le0,
 \end{eqnarray*}
 from which we obtain
 $$ \frac{\p \varphi}{\p t}(t,x)+v(t,x)\cdot\nabla \varphi(t,x)\le 0=\rho(t,x)R(t,x,\nabla \varphi(t,x)).$$
\indent Therefore, we conclude  that $\rho$ is a viscosity subsolution of \eqref{HJ}.
A similar reasoning with
$$\tilde{\rho}_\delta(t,x):=\phi_\delta^0(X(0,t,x))e^{\int_0^t -S(s,X(s,t,x))ds}$$
shows that the function $\tilde{\rho}:[0,T)\times\Omega\to\R$ defined as
$$\tilde{\rho}(t,x):=\phi^0(X(0,t,x))e^{\int_0^t -S(s,X(s,t,x))ds}$$
is a viscosity supersolution of \eqref{HJ} satisfying the initial/boundary condition strictly, where we look at the subdifferential\footnote{
$\dis D^-\tilde{\rho}(z):=\Big\{a\in\R^4\,\Big|\, \liminf_{\tilde{z}\to z}\frac{\tilde{\rho}(\tilde{z})-\tilde{\rho}(z)-  a\cdot(\tilde{z}-z) }{|\tilde{z}-z|}\ge0\Big\}=-D^+(-\tilde{\rho})(z)$ with $z=(t,x)$.} of $\tilde{\rho}$ in Case 3.

{\bf Step 2: application of Perron's method and comparison principle.}
Recall that $\p\Omega$ is invariant under $X$ and that  $S$ is defined to satisfy $S\equiv0$ near $\p\Omega$. Hence, there exist $\ep_1,\ep_2$ with  $0<\ep_2\ll\ep_1\ll\ep$ such that  $S\equiv0$ on $[0,T)\times K_{\ep_1}$ and $X(s,t,x) \in K_{\ep_1}$ for all $s\in[0,t]$ and all $(t,x)\in [0,T)\times K_{\ep_2}$.
Therefore, the viscosity subsolution and supersolution $\rho ,\tilde{\rho}\in C^0([0,T)\times\bar{\Omega})$ are  such that
\begin{eqnarray*}
&&\rho\le\tilde{\rho}\mbox{ in $[0,T)\times\bar{\Omega}$},\quad \rho=\tilde{\rho}=\phi^0(X(0,t,x))\mbox{ in $[0,T)\times K_{\ep_2}$},\\
&&\rho(0,x)=\tilde{\rho}(0,x)=\phi^0(x) \mbox{ on }\bar{\Omega},\quad  \rho=\tilde{\rho}=0\mbox{ on }\bigcup_{0\le s<T}\Big(\{s\}\times\interface(s)\Big).
\end{eqnarray*}
Applying Perron's method (Theorem 3.1 in \cite{Ishii}), we obtain a viscosity solution $\phi:(0,T)\times\Omega\to\R$ of the first equation in \eqref{HJ}  such that
$$\rho\le \phi\le \tilde{\rho} \mbox{ in $(0,T)\times\Omega$}.$$
Since the level-sets of $\rho(t,\cdot),\tilde{\rho}(t,\cdot)$ are equal to $\interface(t)$,  the level-sets of $\phi(t,\cdot)$ must be equal to $\interface(t)$ as well.
It is clear that $\phi$ can be continuously extended up to $([0,T)\times\p\Omega)\cup(\{0\}\times\bar{\Omega})$ satisfying
\begin{eqnarray}\label{boundary}
\rho(t,x)=\phi(t,x)=\tilde{\rho}(t,x)=\phi^0(X(0,t,x))\mbox{ on $([0,T)\times\p\Omega)\cup(\{0\}\times\bar{\Omega})$};
\end{eqnarray}
the same holds for  $\phi^\ast,\phi_\ast$, the  lower/upper semicontinuous envelop of $\phi$:
$$\phi^\ast(t,x)=\phi_\ast(t,x)=\phi^0(X(0,t,x))\mbox{ on $([0,T)\times\p\Omega)\cup(\{0\}\times\bar{\Omega})$}.$$
By definition, we have $\phi_\ast\le \phi^\ast$. On the other hand, $\phi$ being both a viscosity subsolution and supersolution implies that $\phi^\ast$ is an upper semicontinuous viscosity subsolution (note that $(\phi^{\ast})^{\ast}=\phi^\ast$) and $\phi_\ast$ is a lower semicontinuous viscosity supersolution (note that  $(\phi_{\ast})_\ast=\phi_\ast$); the comparison principle (Theorem 8.2 in \cite{CIL}) implies that $\phi^\ast\le \phi_\ast$ on $[0,T)\times\Omega$.
 Here, we remark that \eqref{HJ} in the form of  \eqref{HJ2} does not directly satisfy the monotonicity property ($u\mapsto G(z,u,q)$ must be nondecreasing),  but \eqref{bbb} implies that one can verify the monotonicity property through the change of variable $u=e^{V_0t}\tilde{u}$ (see the following subsection and Chapter 2 of  \cite{Giga}).
Thus, we  conclude that $\phi$ is continuous on $[0,T)\times\bar{\Omega}$ satisfying the initial/boundary condition strictly. Furthermore,  such a viscosity solution is unique. In fact, if $\tilde{\phi}$ is a viscosity solution of \eqref{HJ} in the current sense, the comparison principle implies $\phi\le\tilde{\phi}$ by regarding $\phi$ as a viscosity subsolution and $\tilde{\phi}$ as a viscosity supersolution; $\phi\ge\tilde{\phi}$ by regarding $\phi$ as a viscosity supersolution and $\tilde{\phi}$ as a viscosity subsolution.
\end{proof}
We remark that the result and reasoning of this subsection hold also for the Hamilton-Jacobi equation corresponding to  \eqref{problem2-2}, where we need to add suitable cut-off to the nonlinearity in  \eqref{problem2-2}. To see this, observe that the mapping $u\mapsto u(\beta-|p|)$ is not monotone for all $p\in\R^3$; hence, we define
\begin{align}\label{3simple}
&R(t,x,p):=\eta_1(t,x)\eta_3(|p|)(\beta-|p|),\\\nonumber
&\eta_3(r):=\left\{
\begin{array}{cll}\medskip
1&&\mbox{ for $\dis  0\le r\le 2(\beta+\alpha) $},
\\\nonumber\medskip
0&&\mbox{ for $\dis r\ge 3(\beta+\alpha)$}, \\\medskip
\mbox{monotone smooth transition from $1$ to $0$}&&\mbox{ otherwise},
\end{array}
\right.\\\nonumber
&\mbox{$\beta>\alpha>0$ are the constants given in Subsection 2.3}
\end{align}
to confirm $\sup|R|\le \tilde{V}_0:=2\beta+3\alpha$. Then, with this $R$, Theorem \ref{Thm2} still holds.
\subsection{$C^2$-regularity of viscosity solution near the level-set}
We prove that  the viscosity solution of \eqref{HJ}  coincides with the classical solution of \eqref{problem1} in a $t$-global tubular neighborhood of  the level-set. Note that $|\nabla\phi|$ is nicely controlled on the level-set in \eqref{problem1} and there exists a tubular neighborhood $\Theta$ in which   $\eta_1(t,x)\eta_2(|\nabla\phi(t,x)|)\equiv 1$, i.e., the solution of \eqref{problem1} satisfies  the Hamilton-Jacobi equation in \eqref{HJ} within $\Theta$.

Our proof is based on  the technique known as {\it doubling the number of variables}, which is standard in proofs of  comparison principles for viscosity solutions; more precisely, we adapt the localized version of this technique to our situation with an unusual choice of a penalty function.
The outcome of standard  {\it ``localized  doubling the number of variables''} (see, e.g.,  Theorem 3.12 of \cite{B-CD}) for a Hamilton-Jacobi equation
\begin{align}\label{3HJ-0}
\p_tu+ H(t,x,\nabla u)=0
\end{align}
states the following:
\begin{itemize}
\item[]{\it Let  $u,\tilde{u}$ be viscosity solutions of \eqref{3HJ-0} defined in a cone
$$\mathcal{C}:= \{ (t,x)\in [0,b]\times\R^N\,|\,     |x-z|\le C(b-t)    \},$$
where $C>0$ is a constant such that
\begin{align*}
&|H(t,x,p)-H(t,x,q)|\le C|p-q| ,\\
&|H(t,x,p)-H(s,y,p)|\le C(1+|p|)|(t,x)-(s,y)|.
\end{align*}
If $u(0,\cdot)=\tilde{u}(0,\cdot)$ on $\{x\,|\,|x-z|\le Cb\}$ (the bottom of $\mathcal{C}$), then $u\equiv \tilde{u}$ on $\mathcal{C}$.}
\end{itemize}
\noindent The cone is the region of dependence for general first order Hamilton-Jacobi equations, i.e., the value  $u(b,z)$ is determined by the information only on the bottom of $\mathcal{C}$ (the speed of propagation is finite).
If we directly apply the result to our case, we have to take such a cone contained in the tubular neighborhood $\Theta$, which implies that the time interval $[0,b]$ must be small.
This is the nontrivial aspect of this subsection.
In order to overcome the difficulty, we will introduce an unusual penalty function (i.e., $h(\bar{u}^2)$ below) in {\it localized doubling the number of variables} that consists of the classical solution itself.
Therefore, our technique is specialized for local comparison of a viscosity solution and a classical solution. We emphasize that if both solutions are defined on the whole domain, the issue is obvious, but otherwise not.

Suppose that $v$ satisfies (H1)--(H3).
We consider \eqref{HJ} with initial data $\phi^0\in C^2(\bar{\Omega};\R)$ 
such that $\frac{2}{3}<|\nabla \phi^0|<\frac{4}{3}$ on $\interface(0)$ (this ensures $\eta_1\eta_2\equiv1$ near the level-set).
Then, Theorem \ref{Thm1} yields a tubular neighborhood $\Theta$ of the level-set $\{\interface(t)\}_{t\ge0}$ and a unique $C^2$-solution $\phi:\Theta\to\R$ of \eqref{problem1} with $\eta_1(t,x)\eta_2(|\nabla\phi(t,x)|)\equiv 1$ in $\Theta$ (hence, from here, we rewrite  \eqref{problem1} with $R$ given in \eqref{RRRR}), while Theorem \ref{Thm2}  yields a continuous viscosity solution $\tilde{\phi}:[0,\infty)\times\bar{\Omega}\to\R$ of \eqref{HJ}.
We fix an arbitrary $T>0$ and set $\Theta_T:=\Theta_{0\le t\le T}$.

We need to convert the Hamilton-Jacobi equation in \eqref{HJ} into
\begin{align}\label{3HJ}
\p_tu+ G(t,x,\nabla u, u)=0
\end{align}
with $G$ such that $u\mapsto G(t,x,p,u)$ is nondecreasing.
This is done by the change of the functions as
\begin{align}\label{3convert}
w(t,x):=e^{-V_0t}\phi(t,x),\quad \tilde{w}(t,x):=e^{-V_0t}\tilde{\phi}(t,x)
\end{align}
with the constant $V_0$ given in \eqref{bbb}.
Since $\phi$ is a $C^2$-solution of the original Hamilton-Jacobi equation, it is clear that $w$ satisfies the new Hamilton-Jacobi equation
\begin{align}\label{33HJ}
\p_t u(t,x)+ v(t,x)\cdot\nabla u(t,x)+u(t,x)\Big(V_0-R(t,x,e^{V_0t}\nabla u(t,x))\Big)=0.
\end{align}
Note that $w$ satisfies \eqref{33HJ} also in the sense of viscosity solutions.
In the case of $\tilde{\phi}$ being a viscosity solution, one can also show that $\tilde{w}:=e^{-V_0t}\tilde{\phi}$ is a viscosity solution of \eqref{33HJ}. For the readers' convenience, we briefly explain how to do it: suppose that a test function $\psi$ is such that $\tilde{w}-\psi$ has a local maximum at $(t_0,x_0)$; then, setting the constant $r:=\tilde{w}(t_0,x_0)-\psi(t_0,x_0)$, we have
\begin{align*}
&\tilde{w}(t,x)-\psi(t,x)\le \tilde{w}(t_0,x_0)-\psi(t_0,x_0)=r\mbox{\quad near $(t_0,x_0)$},\\
&\tilde{w}(t,x)-(\psi(t,x)+r)\le \tilde{w}(t_0,x_0)-(\psi(t_0,x_0)+r)=0\mbox{\quad near $(t_0,x_0)$},
\end{align*}
 from which we obtain
 \begin{align*}
&e^{V_0t}\tilde{w}(t,x)-e^{V_0t}(\psi(t,x)+r)\le 0=\tilde{w}(t_0,x_0)-(\psi(t_0,x_0)+r)\mbox{\quad near $(t_0,x_0)$},\\
&e^{V_0t}\tilde{w}(t,x)-e^{V_0t}(\psi(t,x)+r)\le 0=e^{V_0t_0}\tilde{w}(t_0,x_0)-e^{V_0t_0}(\psi(t_0,x_0)+r)\mbox{\quad near $(t_0,x_0)$},\\
&\tilde{\phi}(t,x)-e^{V_0t}(\psi(t,x)+r)\le \tilde{\phi}(t_0,x_0)-e^{V_0t_0}(\psi(t_0,x_0)+r)\mbox{\quad near $(t_0,x_0)$};
\end{align*}
since $\tilde{\phi}$ is a viscosity subsolution, it holds that
\begin{align*}
&\p_t\{e^{V_0t}(\psi(t,x)+r)\}+v(t,x)\cdot\nabla\{e^{V_0t}(\psi(t,x)+r)\} \\
&\qquad - \tilde{\phi}(t,x)R\Big(t,x,\nabla\{e^{V_0t}(\psi(t,x)+r)\}\Big)|_{t=t_0,x=x_0}\\
& =e^{V_0t_0}V_0(\psi(t_0,x_0)+r) + e^{V_0t_0}\p_t\psi(t_0,x_0)+e^{V_0t_0}v(t_0,x_0)\cdot \nabla\psi(t_0,x_0) \\
&\qquad - \tilde{\phi}(t_0,x_0)R\Big(t_0,x_0,e^{V_0t_0}\nabla\psi(t_0,x_0)\}\Big)    \le 0;
\end{align*}
hence, noting that  $\psi(t_0,x_0)+r=\tilde{w}(t_0,x_0)$ and $\tilde{\phi}(t_0,x_0)=e^{V_0t_0}\tilde{w}(t_0,x_0)$, we obtain
\begin{align*}
\p_t \psi(t_0,x_0)+v(t_0,x_0)\cdot\nabla \psi(t_0,x_0)+\tilde{w}(t_0,x_0)\Big(V_0 -R(t_0,x_0,e^{V_0t_0}\nabla\psi(t_0,x_0)) \Big)\le 0;
\end{align*}
therefore, we conclude that $\tilde{w}$ is a viscosity subsolution of \eqref{33HJ}; similar argument shows that $\tilde{w}$ is a viscosity supersolution of \eqref{33HJ}.

We rewrite \eqref{33HJ}
in the form of \eqref{3HJ} with
$$G(t,x,p,u):=v(t,x)\cdot p+ u\tilde{R}(t,x,p),\quad \tilde{R}(t,x,p):=  (V_0-R(t,x,e^{V_0t}p)).$$
Due to the condition of $v$ and the definition of $R$ (see \eqref{RRRR}), there exists a constant $C>0$ such that
\begin{align}\label{3con}
\begin{array}{lll}
&|\tilde{R}(t,x,p)-\tilde{R}(t,x,q)|\le C|p-q| ,\quad \forall\,(t,x)\in\Theta_T,\,\,\, \forall\,p,q\in \R^3,\\
&|\tilde{R}(t,x,p)-\tilde{R}(s,y,p)|\le C(1+|p|)|(t,x)-(s,y)|,\,\,\, \forall\,(t,x)\in\Theta_T,\,\,\, \forall\,p\in \R^3.
\end{array}
\end{align}
For a technical reason, we consider another tubular neighborhood  $\tilde{\Theta}\subset\Theta$  of $\{\interface(t)\}_{t\ge0}$ such that $\phi$ is defined on the closure of  $\tilde{\Theta}$, and
(partly) describe $\tilde{\Theta}_T:=\tilde{\Theta}|_{0\le t \le T}$  as a family of cylinders, i.e., a foliation:
\begin{itemize}
\item Let  $\alpha>0$ be a constant such that
\begin{align}\label{3alpha}
\alpha-2V_0-C\sup_{\tilde{\Theta}_T}|\nabla w|>0\mbox{\,\,\, with }   w(t,x)=e^{-V_0t}\phi(t,x);
\end{align}
\item Consider the function
$$\bar{u}(t,x):= e^{\alpha t}w(t,x) ,$$
where $\bar{u}$ solves in the classical sense
\begin{align}\label{3pen}
\p_t u+v\cdot\nabla u+u\tilde{R}(t,x,e^{-\alpha t}\nabla u)= \alpha u ,\quad u(0,\cdot)=\phi^0;
\end{align}
\item With a constant $m_0>0$, define
\begin{align*}
&A_m:=\{ (t,x)\in[0,T]\times\Omega \,|\,\bar{u}(t,x)=m\},\quad-m_0\le m\le m_0,\quad  \\
&\Gamma_T:= \bigcup_{-m_0\le m\le m_0} A_m.
\end{align*}
Since $\nabla \bar{u}\neq0$ near $\cup_{0\le t\le T} (\{t\}\times\interface(t))$, we may choose  $m_0>0$ (possibly very small) so that  $\Gamma_T$ is contained in $\tilde{\Theta}_T$ and $\p\Gamma_T\setminus \Gamma_T|_{t=0,T}=(A_{m_0}\cup A_{-m_0})|_{0<t<T}$, while $\Gamma_T$ contains the $[0,T]$-part of  a tubular neighborhood of  $\{\interface(t)\}_{0\le t\le T}$. Note that such an $m_0$ depends, in general, on $T$.
\end{itemize}
\begin{Thm}\label{Theorem3}
It holds that $\phi\equiv \tilde{\phi}$ on $\Gamma_T$, where $T>0$ is arbitrary.
\end{Thm}
\medskip

\noindent{\bf Remark.} {\it $\Gamma_T$ would become narrower  as $T$ gets larger. In order to find a more optimal region on which $\phi\equiv\tilde{\phi}$, one can iterate Theorem \ref{Theorem3} for $T=T_0$, $T=2T_0$, $T=3T_0$, $\ldots$, with a small $T_0>0$.}

\medskip
\begin{proof}[{\it Proof of Theorem \ref{Theorem3}.}]
We proceed by contradiction. Suppose that the assertion does not hold. We assume
$$\mbox{$\dis\max_{\Gamma_T}(\phi-\tilde{\phi})>0$ (in the other case, we switch $\phi$ and $\tilde{\phi}$).}$$
From now on, we deal with $w,\tilde{w}$ given as \eqref{3convert}.
Then, we find an interior point  $(t^\ast,x^\ast)$ of $\Gamma_T$ such that
$$\sigma:=w(t^\ast,x^\ast)-\tilde{w}(t^\ast,x^\ast)>0.$$
Since the level-set of $w$ and that of $\tilde{w}$ are identical, we see that $(t^\ast,x^\ast)\not\in A_0=\cup_{0\le t\le T}(\{t\}\times\interface(t))$. Let $m^\ast\in(-m_0,m_0) $ be  such that
$$\bar{u}(t^\ast,x^\ast)=m^\ast,\,\,\,\mbox{or equivalently, }(t^\ast,x^\ast)\in A_{m^\ast}.$$
Let $\delta>0$ be a constant such that
$$
(m^\ast)^2+2\delta<(m_0)^2.$$
 Take a constant $M>0$ such that
$$M\ge\max_{(t,x,s,y)\in \Gamma_T\times\Gamma_T} |w(t,x)-\tilde{w}(s,y)|$$
and a monotone increasing $C^1$-function $h: \R\to [0,3M]$ such that
\begin{align*}
h(r)
=\left\{
\begin{array}{cll}
 &3M, &\mbox{ if  $(m^\ast)^2 + 2\delta \le r$},\\
 &0,&\mbox{ if $ r\le (m^\ast)^2+\delta$},\\
 &\mbox{monotone transition between $0$ and $3M$},& \mbox{ otherwise}.
\end{array}
\right.
\end{align*}
For each $\ep>0$, $\lambda>0$ (they are independent of each other; $\lambda>0$ will be appropriately fixed at some point, while $\ep$ will be sent to $0$ at the end), define the function $F_{\ep,\lambda}:\Gamma_T\times\Gamma_T\to\R$ as
\begin{align*}
F_{\ep,\lambda}(t,x,s,y):=&w(t,x)-\tilde{w}(s,y)-\lambda(t+s)-\frac{1}{\ep^2}\Big(|x-y|^2+|t-s|^2\Big)\\
&- h(\bar{u}(t,x)^2)-h(\bar{u}(s,y)^2).
\end{align*}
 Let $(t_0,x_0,s_0,y_0)=(t_0(\ep,\lambda),x_0(\ep,\lambda),s_0(\ep,\lambda),y_0(\ep,\lambda))\in \Gamma_T\times \Gamma_T$ be such that
\begin{align*}
F_{\ep,\lambda}(t_0,x_0,s_0,y_0)=\max_{ \Gamma_T\times \Gamma_T}F_{\ep,\lambda}.
\end{align*}
\indent We first claim that there exists a  sufficiently small $\lambda>0$ for which we have
 \begin{align}\label{322-1}
& |x_0-y_0|=O(\ep),\,\, |t_0-s_0|=O(\ep)\mbox{ \quad as $\ep\to0+$},\\\label{322-12}
& \mbox{$t_0=0$ or $s_0=0$ or $(t_0,x_0,s_0,y_0)\in\,$[interior of $\Gamma_T\times \Gamma_T$],\,\,\,\, for all $0<\ep\ll1$}.
\end{align}
In fact, observe that
\begin{align*}
F_{\ep,\lambda}(t_0,x_0,s_0,y_0)\ge F_{\ep,\lambda}(t^\ast,x^\ast,t^\ast,x^\ast),
\end{align*}
that is,
\begin{align*}
&w(t_0,x_0)-\tilde{w}(s_0,y_0)-\lambda(t_0+s_0)
-\frac{1}{\ep^2}\Big(|x_0-y_0|^2+|t_0-s_0|^2\Big)\\
&\qquad- h(\bar{u}(t_0,x_0)^2)-h(\bar{u}(s_0,y_0)^2)
 \ge w(t^\ast,x^\ast)-\tilde{w}(t^\ast,x^\ast)-2\lambda t^\ast-2h(\bar{u}(t^\ast,x^\ast)^2).
\end{align*}
Hence, with the definition of $h$, we may fix $\lambda>0$ small enough  to obtain for all sufficiently small $\ep>0$,
\begin{align*}
&\frac{1}{\ep^2}\Big(|x_0-y_0|^2+|t_0-s_0|^2\Big)
+ h(\bar{u}(t_0,x_0)^2)+h(\bar{u}(s_0,y_0)^2)\\
&\quad \le w(t_0,x_0)-\tilde{w}(s_0,y_0)-(w(t^\ast,x^\ast)-\tilde{w}(t^\ast,x^\ast))
+2\lambda t^\ast+2h((m^\ast)^2)\\
&\quad =w(t_0,x_0)-\tilde{w}(s_0,y_0)-(w(t^\ast,x^\ast)-\tilde{w}(t^\ast,x^\ast))
+2\lambda t^\ast < 3M,
\end{align*}
and
\begin{align}\label{831}
F_{\ep,\lambda}(t_0,x_0,s_0,y_0)
&\ge w(t^\ast,x^\ast)-\tilde{w}(t^\ast,x^\ast)-2\lambda t^\ast-2h(\bar{u}(t^\ast,x^\ast)^2)\\\nonumber
&=\sigma  -2\lambda t^\ast
\ge \frac{\sigma}{2}>0.
\end{align}
Since $h$ is nonnegative,  \eqref{322-1} is clear.
Suppose that $(t_0,x_0)$ is on the ``lateral surface'' of $\Gamma_T$, i.e., $(t_0,x_0)\in A_{\pm m_0}$, or equivalently, $|\bar{u}(t_0,x_0)|=m_0$.
Then, with the definition of $h$, we see that
\begin{align*}
F_{\ep,\lambda}(t_0,x_0,s_0,y_0)
&=w(t_0,x_0)-\tilde{w}(s_0,y_0)-\lambda(t_0+s_0)
-\frac{1}{\ep^2}\Big(|x_0-y_0|^2+|t_0-s_0|^2\Big)\\
&\quad - h(m_0^2)-h(\bar{u}(s_0,y_0)^2)
\le M-h(m_0^2)=-2M <0,
\end{align*}
which contradicts to \eqref{831}; the same to $(s_0,y_0)$.  Thus, \eqref{322-12} is confirmed.

 Next, we sharpen the estimates \eqref{322-1} to be of small order. Since $F_{\ep,\lambda}(t_0,x_0,s_0,y_0)\ge F_{\ep,\lambda}(t_0,x_0,t_0,x_0)$, we have
\begin{align*}
&w(t_0,x_0)-\tilde{w}(s_0,y_0)-\lambda(t_0+s_0)
-\frac{1}{\ep^2}\Big(|x_0-y_0|^2+|t_0-s_0|^2\Big)\\
&\quad - h(\bar{u}(t_0,x_0)^2)-h(\bar{u}(s_0,y_0))
 \ge w(t_0,x_0)-\tilde{w}(t_0,x_0)-2\lambda t_0
 - 2h(\bar{u}(t_0,x_0)^2).
\end{align*}
Hence, by  \eqref{322-1} and the continuity of $\tilde{w}$ and $h$, we obtain
\begin{align*}
&\frac{1}{\ep^2}\Big(|x_0-y_0|^2+|t_0-s_0|^2\Big)
 \le \tilde{w}(t_0,x_0)-\tilde{w}(s_0,y_0)+\lambda (t_0-s_0) \\
 &\qquad +h(\bar{u}(t_0,x_0)^2)-h(\bar{u}(s_0,y_0)^2)
\to0\quad \mbox{as $\ep\to0+$},
\end{align*}
which implies
\begin{align}\label{323-1}
  |x_0-y_0|=o(\ep),\,\,\, |t_0-s_0|=o(\ep)\mbox{ \quad as $\ep\to0+$}.
\end{align}
\indent We claim that the case $t_0=0$ or $s_0=0$ is impossible in \eqref{322-12} for all sufficiently small $\ep>0$.  In fact, suppose that $t_0=0$. By \eqref{831}, \eqref{323-1} and $w(0,\cdot)\equiv\tilde{w}(0,\cdot)$ on $\Gamma_T|_{t=0}$, we see that
\begin{align*}
\frac{\sigma}{2}&\le w(t_0,x_0)-\tilde{w}(s_0,y_0)= \tilde{w}(t_0,x_0)-\tilde{w}(s_0,y_0)\le \tilde{\omega}(|t_0-s_0|+|x_0-y_0|)=\tilde{\omega}(o(\ep)),
\end{align*}
where $\tilde{\omega}(\cdot)$ stands for the modulus of continuity of $\tilde{w}$.
If $\ep>0$ is sufficiently small, we face a contradiction.
The same to $s_0$.

We confirmed that the point $(t_0,x_0)$ is an interior point of $\Gamma_T$, or $t_0=T$ and $x_0$ is an  interior point of $\Gamma_T|_{t=T}$; $( s_0,y_0)$ is an interior point of $\Gamma_T$, or $s_0=T$  and $y_0$ is an  interior point of $\Gamma_T|_{t=T}$,  which enables us to test $w$ and $\tilde{w}$ at $(t_0,x_0)$, $( s_0,y_0)$ in terms of viscosity sub-/supersolutions (see Lemma in Section 10.2 of \cite{Evans-book} for a remark on the case $t_0=T$ or $s_0=T$).
Observe that the mapping $(t,x)\mapsto F_{\ep,\lambda}(t,x,s_0,y_0)$
takes the maximum at the point $(t_0,x_0)$; then,
introducing the $C^1$-function $\psi$ as
\begin{align*}
\psi(t,x)&:= \tilde{w}(s_0,y_0)+\lambda(t+s_0)
+\frac{1}{\ep^2}\Big(|x-y_0|^2+|t-s_0|^2\Big)
+h(\bar{u}(t,x)^2)
+h(\bar{u}(s_0,y_0)^2),
\end{align*}
we see that the definition of $F_{\ep,\lambda}$ implies that
$$w-\psi\mbox{ takes a maximum at $(t_0,x_0)$}.$$
Since $w$ satisfies the equation in the sense of viscosity subsolutions, we have
$$\p_t \psi(t_0,x_0)+G(t_0,x_0,\nabla_x \psi(t_0,x_0),w(t_0,x_0))\le 0.$$
Therefore, we obtain
\begin{align}\label{3sub}
&\lambda +\frac{2}{\ep^2}(t_0-s_0)
+2h'(\bar{u}(t_0,x_0)^2)\bar{u}(t_0,x_0)\p_t \bar{u}(t_0,x_0)\\\nonumber
& +G\Big(t_0,x_0, \frac{2}{\ep^2}(x_0-y_0)
+2h'(\bar{u}(t_0,x_0)^2)\bar{u}(t_0,x_0)\nabla \bar{u}(t_0,x_0),w(t_0,x_0)  \Big)\le 0.
\end{align}
Similarly,  the mapping $(s,y)\mapsto -F_{\ep,\lambda}(t_0,x_0,s,y)$
takes a minimum at the point $(s_0,y_0)$; then,
introducing the $C^1$-function $\tilde{\psi}$ as
\begin{align*}
\tilde{\psi}(s,y)&:= w(t_0,x_0)-\lambda(t_0+s)
-\frac{1}{\ep^2}\Big(|x_0-y|^2+|t_0-s|^2\Big)
 -h(\bar{u}(t_0,x_0)^2)-h(\bar{u}(s,y)^2),
\end{align*}
we see that
$$\tilde{w}-\tilde{\psi}\mbox{ takes a minimum at $(s_0,y_0)$}.$$
Since $\tilde{w}$ satisfies the equation in the sense of viscosity supersolutions, we have
$$\p_s \tilde{\psi}(s_0,y_0)+G(s_0,y_0,\nabla_y \tilde{\psi}(s_0,y_0),\tilde{w}(s_0,y_0))\ge 0.$$
Consequently, we obtain
\begin{align}\label{3super}
&-\lambda +\frac{2}{\ep^2}(t_0-s_0)
-2h'(\bar{u}(s_0,y_0)^2)\bar{u}(s_0,y_0)\p_tu(s_0,y_0)\\\nonumber
& +G\Big(s_0,y_0, \frac{2}{\ep^2}(x_0-y_0)
-2h'(\bar{u}(s_0,y_0)^2)\bar{u}(s_0,y_0)\nabla \bar{u}(s_0,y_0),\tilde{w}(s_0,y_0)  \Big)\ge 0.
\end{align}
It follows from \eqref{831} that
\begin{align*}
w(t_0,x_0)-\tilde{w}(s_0,x_0)> 0.
\end{align*}
Since $G(t,x,p,u)$ is nondecreasing with respect to $u$, \eqref{3super} yields
\begin{align}\label{3super2}
&-\lambda +\frac{2}{\ep^2}(t_0-s_0)
-2h'(\bar{u}(s_0,y_0)^2)\bar{u}(s_0,y_0)\p_tu(s_0,y_0)\\\nonumber
& +G\Big(s_0,y_0, \frac{2}{\ep^2}(x_0-y_0)
-h'(\bar{u}(s_0,y_0)^2)\bar{u}(s_0,y_0)\nabla u(s_0,y_0),w(t_0,x_0)  \Big)\ge 0.
\end{align}
By \eqref{3sub} and \eqref{3super2}, we obtain
\begin{align*}
2\lambda
&\le -2h'(\bar{u}(s_0,y_0)^2)\bar{u}(s_0,y_0)\Big(\p_t\bar{u}(s_0,y_0)
+v(s_0,y_0)\cdot \nabla \bar{u}(s_0,y_0)\Big)\\
&\quad -2h'(\bar{u}(t_0,x_0)^2)\bar{u}(t_0,x_0)\Big(\p_t\bar{u}(t_0,x_0)
+v(t_0,x_0)\cdot \nabla \bar{u}(t_0,x_0)\Big)\\
&\quad + 2(v(s_0,y_0)-v(t_0,x_0))\cdot\frac{x_0-y_0}{\ep^2}\\
&\quad + w(t_0,x_0)\tilde{R}\Big(s_0,y_0, \frac{2}{\ep^2}(x_0-y_0)
-2h'(\bar{u}(s_0,y_0)^2)\bar{u}(s_0,y_0)\nabla \bar{u}(s_0,y_0)
\Big)\\
&\quad -w(t_0,x_0)\tilde{R}\Big(t_0,x_0, \frac{2}{\ep^2}(x_0-y_0)+2h'(\bar{u}(t_0,x_0)^2)\bar{u}(t_0,x_0)\nabla \bar{u}(t_0,x_0)
 \Big).
\end{align*}
Since $\bar{u}$ solves \eqref{3pen} and $|\tilde{R}|\le 2V_0$, we obtain with \eqref{3con}, $h'\ge0$ and \eqref{323-1},
\begin{align*}
&2\lambda
\le -2h'(\bar{u}(s_0,y_0)^2)\bar{u}(s_0,y_0) \cdot \bar{u}(s_0,y_0)\Big(\alpha-\tilde{R}(s_0,y_0,e^{-\alpha s_0}\nabla \bar{u}(s_0,y_0))\Big)\\
&\quad  -2h'(\bar{u}(t_0,x_0)^2)\bar{u}(t_0,x_0) \cdot \bar{u}(t_0,x_0)\Big(\alpha-\tilde{R}(t_0,x_0,e^{-\alpha t_0}\nabla \bar{u}(t_0,x_0))\Big)\\
&\quad + 2\mbox{[Lipschitz constant of $v$ within $\Theta_T$]} |(t_0,x_0)-(s_0,y_0)| \frac{|x_0-y_0|}{\ep^2}\\
&\quad +2C|w(t_0,x_0)|| h'(\bar{u}(s_0,y_0)^2)\bar{u}(s_0,y_0)\nabla \bar{u}(s_0,y_0)+h'(\bar{u}(t_0,x_0)^2)\bar{u}(t_0,x_0)\nabla \bar{u}(t_0,x_0)|
\\
&\quad + C|w(t_0,x_0)|\Big(1+ \Big|  \frac{2}{\ep^2}(x_0-y_0)
+2h'(\bar{u}(t_0,x_0)^2)\bar{u}(t_0,x_0)\nabla \bar{u}(t_0,x_0)  \Big| \Big)|(t_0,x_0)-(s_0,x_0)|\\
&\le   -2e^{2\alpha t_0}h'(\bar{u}(s_0,y_0)^2)w(t_0,x_0) ^2\Big(\alpha-2V_0\Big)+o(\ep) \\
&\quad -2e^{2\alpha t_0}h'(\bar{u}(t_0,x_0)^2)w(t_0,x_0)^2\Big(\alpha-2V_0\Big)\\
&\quad +2e^{2\alpha t_0}C h'(\bar{u}(s_0,y_0)^2)w(t_0,x_0)^2|\nabla w(s_0,y_0)|+o(\ep)\\
&\quad+2e^{2\alpha t_0}Ch'(\bar{u}(t_0,x_0)^2)w(t_0,x_0)^2|\nabla w(t_0,x_0)|+\frac{o(\ep)^2}{\ep^2}\\
&= -2e^{2\alpha t_0}h'(\bar{u}(s_0,y_0)^2)w(t_0,x_0)^2\Big( \alpha-2V_0 - C |\nabla w(s_0,y_0)|   \Big)\\
&\quad -2e^{2\alpha t_0}h'(\bar{u}(t_0,x_0)^2)w(t_0,x_0)^2\Big( \alpha-2V_0 - C |\nabla w(t_0,x_0)|   \Big)+\frac{o(\ep)^2}{\ep^2}.
\end{align*}
The right-hand side of this inequality is bounded from above by $\frac{o(\ep)^2}{\ep^2}$ due to the choice of $\alpha$ in \eqref{3alpha}. Thus, we reach a contradiction by sending $\ep\to0+$.
\end{proof}

We remark that the result and reasoning of this subsection hold also for the problems \eqref{problem2-2} and \eqref{HJ} with $R$ given in \eqref{3simple} and $C^2$-initial data $\phi^0$ such that $\beta-\alpha\le |\nabla \phi^0|\le \beta+\alpha$.


\medskip\medskip\medskip

\noindent{\bf Acknowledgement.}
The authors thank Prof. Qing Liu in Okinawa Institute of Science and Technology for his pointing out Theorem 3.12 of \cite{B-CD}.
Dieter Bothe and Mathis Fricke are supported by the Deutsche Forschungsgemeinschaft (DFG, German
Research Foundation) -- Project-ID 265191195 -- SFB 1194.
This work was done mostly during Kohei Soga's one-year research stay at the department of mathematics, Technische Universit\"at Darmstadt, Germany, with the grant Fukuzawa Fund (Keio Gijuku Fukuzawa Memorial Fund for the Advancement of Education and Research).  Kohei Soga is also supported by JSPS Grant-in-aid for Young Scientists \#18K13443 and JSPS Grants-in-Aid for Scientific Research (C) \#22K03391.
\medskip\medskip\medskip

\noindent{\bf Data availability.}
Data sharing not applicable to this article as no datasets were generated or analysed during the current study.

\medskip

\noindent{\bf Conflicts of interest statement.} The authors state that there is no conflict of interest.
\medskip

\renewcommand{\theequation}{A.\arabic{equation} }
\appendix
\def\thesection{Appendix 1: method of characteristics}
\section{}
\setcounter{equation}{0}

We explain the method of characteristics for Hamilton-Jacobi equations.
Consider a function $H=H(t,x,p,\Phi)$ such that
\begin{align*}
&H:[0,\infty)\times\R^N\times\R^N\times\R\to\R \mbox{ is $C^1$ in $(t,x,p,\Phi)$};\\
&\mbox{$H$ is twice partial differentiable in $(x,p,\Phi)$;}\\  
&\mbox{ the partial derivatives of $H$ are continuous  in $(t,x,p,\Phi)$}.
\end{align*}
Note that the upcoming argument is available also for a function $H$ defined on  $[0,\infty)\times K$ with an open subset $K\subset \R^{2N+1}$. Let $w:\R^n\to\R$ be a given $C^2$-function. We will construct a unique $C^2$-solution $u$  of
\begin{eqnarray}\label{gHJ}
\frac{\p u}{\p t}+H(t,x,\nabla u,u)=0 \mbox{ in $O$},\quad u(0,x)=w(x)\mbox{ on $O|_{t=0}$},
\end{eqnarray}
where $O\subset\R\times\R^N$ is a neighborhood of a point $(0,\xi)\in\R\times\R^N$. We remark that $\xi$ can be arbitrary, but the size of $O$ depends on $\xi$; there exists $\ep>0$ such that $[-\ep,\ep]\times B_\ep (\xi)\subset O$ (here, $B_\ep (\xi)\subset\R^N$ is the $\ep$-ball with the center $\xi$) and \eqref{gHJ} is solvable for negative time; $O$ is determined by the inverse map theorem around the point $(0,\xi)$; if one wants to solve \eqref{gHJ} in $O$ larger than an open set coming from the inverse map theorem around a single point, one needs more arguments, in addition to the upcoming discussion, to confirm injectivity.

We consider autonomization by introducing
$$\tilde{H}(t,x,r,p,\Phi):=r+H(t,x,p,\Phi)$$
 with the extended configuration variable $(t,x)$. Then, we see that \eqref{gHJ} can be seen as
$$\tilde{H}(t,x,\mbox{$\frac{\p u}{\p t}$},\nabla u,u)=0.$$
The characteristic ODE-system of \eqref{gHJ} is given as  the autonomous (contact) Hamiltonian system generated by $\tilde{H}(t,x,r,p,\Phi)$, i.e., setting $Q:=(t,x)$, $P=(r,p)$, $\tilde{H}=\tilde{H}(Q,P,\Phi)$,
\begin{eqnarray*}
&&Q'(s)=\frac{\p \tilde{H}}{\p P}(Q(s),P(s),\Phi(s)), \\
&&P'(s)=-\frac{\p \tilde{H}}{\p Q}(Q(s),P(s),\Phi(s))-\frac{\p \tilde{H}}{\p \Phi}(Q(s),P(s),\Phi(s))P(s),\\
&&\Phi'(s)=\frac{\p \tilde{H}}{\p P}(Q(s),P(s),\Phi(s))\cdot P(s)-\tilde{H}(Q(s),P(s),\Phi(s)),\\
&&Q(0)=(0,\xi),\,\,\,\,P(0)=(-H(0,\xi,\nabla w(\xi),w(\xi)), \nabla w(\xi)),\,\,\,\,\Phi(0)=w(\xi),\quad \xi\in\R^N,
\end{eqnarray*}
where the solution is written as $Q(s;\xi),P(s;\xi),\Phi(s;\xi)$.
This system is equivalent to
\begin{eqnarray}\label{CH1}
&&x'(s)=\frac{\p H}{\p p}(t(s),x(s),p(s),\Phi(s)),\\\nonumber &&x(0)=\xi,\\\label{CH2}
&&p'(s)=-\frac{\p H}{\p x}(t(s),x(s),p(s),\Phi(s))-\frac{\p H}{\p \Phi}(t(s),x(s),p(s),\Phi(s))p(s),\\\nonumber
&& p(0)=\nabla w(\xi),\\\label{CH3}
&&\Phi'(s)=\frac{\p H}{\p p}(t(s),x(s),p(s),\Phi(s))\cdot p(s)-H(t(s),x(s),p(s),\Phi(s)) ,\\\nonumber
&&\Phi(0)=w(\xi),\\\label{CH4}
&&t'(s)=1,\quad\\\nonumber
&& t(0)=0,\\\label{CH5}
&&r'(s)=-\frac{\p H}{\p t} (t(s),x(s),p(s),\Phi(s))-\frac{\p H}{\p \Phi}(t(s),x(s),p(s),\Phi(s))r(s),\\\nonumber
&&r(0)=-H(0,\xi,\nabla w(\xi), w(\xi)).
\end{eqnarray}
Observe that
\begin{eqnarray}\label{A10}
t(s;\xi)&\equiv& s,\quad \forall\,\xi,\,\, \forall\,s,\\\label{A11}
\tilde{H}(Q(0;\xi),P(0;\xi),\Phi(0;\xi))&=&0,\quad \forall\,\xi,
\\\label{A12}
\frac{\p}{\p s}\{\tilde{H}(Q(s;\xi),P(s;\xi),\Phi(s;\xi))\}
&=&\frac{\p \tilde{H}}{\p Q}(Q(s;\xi),P(s;\xi),\Phi(s;\xi))\cdot Q'(s;\xi)\\\nonumber
&&\!\!\!\!\!\!\!\!\!\!\!\!\!\!\!\!\!\!\!\!\!\!\!\!\!\!\!\!\!\!\!\!\!\!\!\!\!\!\!\!\!\!\!\!\!\!\!\!\!\!\!\!\!\!\!\!\!\!\!\!\!\!\!\!\!\!\!\!\!\!\!\!\!\!\!\!\!\!\!\!\!\!\!\!\!\!\!\!\!\!\!\!\!\!\!\!+\frac{\p \tilde{H}}{\p P}(Q(s;\xi),P(s;\xi),\Phi(s;\xi))\cdot P'(s;\xi)+\frac{\p \tilde{H}}{\p \Phi}(Q(s;\xi),P(s;\xi),\Phi(s;\xi))\Phi'(s)
\\\nonumber
&&\!\!\!\!\!\!\!\!\!\!\!\!\!\!\!\!\!\!\!\!\!\!\!\!\!\!\!\!\!\!\!\!\!\!\!\!\!\!\!\!\!\!\!\!\!\!\!\!\!\!\!\!\!\!\!\!\!\!\!\!\!\!\!\!\!\!\!\!\!\!\!\!\!\!\!\!\!\!\!\!\!\!\!\!\!\!\!\!\!\!\!\!\!\!\!\!=-\tilde{H}(Q(s;\xi),P(s;\xi),\Phi(s;\xi))\frac{\p \tilde{H}}{\p \Phi}(Q(s;\xi),P(s;\xi),\Phi(s;\xi)).
\end{eqnarray}
\eqref{A10} implies that $t(s)=s$ serves as the non-autonomous factor in  \eqref{CH1}, \eqref{CH2},\eqref{CH3},\eqref{CH5};
hence, our regularity assumption on $H$ is sufficient to provide the smooth dependency with respect to initial  data (note that this is apparently not clear if we only  look at the autonomous system with $\tilde{H}(Q,P,\Phi)$, as $C^2$-regularity in $Q$ is missing).
\eqref{A11} and \eqref{A12} imply  that
\begin{eqnarray}\label{A13}
&\tilde{H}(Q(s;\xi),P(s;\xi),\Phi(s;\xi))\equiv 0 \mbox{ as long as $Q(s;\xi),P(s;\xi),\Phi(s;\xi)$ exist}.
\end{eqnarray}
The map $F(s,\xi):=Q(s;\xi)=(s,x(s;\xi))$ is $C^1$-smooth  as long as the characteristic ODEs have solutions. Furthermore, it holds that
\begin{eqnarray}\label{inverse}
F(0,\xi)=(0,\xi),\quad \det D_{(s,\xi)}F(0,\xi)=1,
\end{eqnarray}
where $D_{(s,\xi)}F$ is the Jacobian matrix of $F$.
Therefore, for each $\xi\in\R^N$, the inverse map theorem guarantees that there exist two sets, a neighborhood $O_1$ of $(0,\xi)$ and a neighborhood $O_2$ of $F(0,\xi)=(0,\xi)$, such that $F:O_1\to O_2$ is a $C^1$-diffeomorphism.  We obtain the inverse map of $F$ as
$$F^{-1}:O_2\to O_1,\quad F^{-1}(t,x)=(t,\varphi(t,x)),$$
where $\xi=\varphi(t,x)$ is the point for which $x(s;\xi)$ passes through $x$ at $s=t$.
The essential point is to obtain open sets $O_1$ and $O_2=F(O_1)$ such that $F:O_1\to O_2$ is a $C^1$-diffeomorphism,  no matter how we find them.
The easiest way is to use  the inverse map theorem around a single point $(0,\xi)$ based on \eqref{inverse}.
The invertibility of $F$, or the invertibility of $x(s;\xi)$ for fixed $s$, on a wider region requires injectivity, which will be an additional issue to be verified.

We proceed, assuming that $O_1\subset\R\times\R^N$ is a neighborhood of a subset of $\{0\}\times\R^N$ and $F:O_1\to O_2=F(O_1)$ is a $C^1$-diffeomorphism.
By  \eqref{A13}, we have
$$\frac{\p \Phi}{\p s}(s;\xi)=\frac{\p \tilde{H}}{\p P}(Q(s;\xi),P(s;\xi),\Phi(s;\xi))\cdot P(s;\xi)=P(s;\xi)\cdot \frac{\p Q}{\p s}(s;\xi).$$
It holds that
$$\frac{\p \Phi}{\p \xi_i}(s;\xi)=P(s;\xi)\cdot \frac{\p Q}{\p \xi_i}(s;\xi),\quad i=1,2,\ldots N.$$
In fact, we have
\begin{eqnarray*}
&&\frac{\p \Phi}{\p \xi_i}(0;\xi)-P(0;\xi)\cdot \frac{\p Q}{\p \xi_i}(0;\xi)=\frac{\p w}{\p \xi_i}(\xi)-\nabla w(\xi) \cdot \frac{\p \xi}{\p \xi_i}=0,\\
&&\frac{\p}{\p s}\Big( \frac{\p \Phi}{\p \xi_i}(s;\xi)-P(s;\xi)\cdot \frac{\p Q}{\p \xi_i}(s;\xi) \Big)
= \frac{\p}{\p s}\frac{\p \Phi}{\p \xi_i}(s;\xi)- \frac{\p}{\p s}\Big(P(s;\xi)\cdot \frac{\p Q}{\p \xi_i}(s;\xi)\Big)\\
&&\quad =\frac{\p}{\p \xi_i}\frac{\p \Phi}{\p s}(s;\xi)- \frac{\p}{\p s}\Big(P(s;\xi)\cdot \frac{\p Q}{\p \xi_i}(s;\xi)\Big)\\
&&\quad =\frac{\p}{\p \xi_i}\Big(P(s;\xi)\cdot \frac{\p Q}{\p s}(s;\xi)\Big)- \frac{\p}{\p s}\Big(P(s;\xi)\cdot \frac{\p Q}{\p \xi_i}(s;\xi)\Big)\\
&&\quad =\frac{\p Q}{\p s}(s;\xi)\cdot  \frac{\p P}{\p \xi_i}(s;\xi)-  \frac{\p P}{\p s}(s;\xi)\frac{\p Q}{\p \xi_i}(s;\xi)\\
&&\quad=\frac{\p\tilde{H}}{\p P}(Q(s;\xi),P(s;\xi),\Phi(s;\xi)) \cdot  \frac{\p P}{\p \xi_i}(s;\xi)+\frac{\p\tilde{H}}{\p Q}(Q(s;\xi),P(s;\xi),\Phi(s;\xi)) \cdot \frac{\p Q}{\p \xi_i}(s;\xi)\\
&&\qquad +\frac{\p\tilde{H}}{\p \Phi}(Q(s;\xi),P(s;\xi),\Phi(s;\xi))P(s;\xi)\cdot   \frac{\p Q}{\p \xi_i}(s;\xi)\\
&&\quad = \frac{\p }{\p \xi_i}\tilde{H}(Q(s;\xi),P(s;\xi),\Phi(s;\xi)) \\
&&\qquad -\frac{\p\tilde{H}}{\p \Phi}(Q(s;\xi),P(s;\xi),\Phi(s;\xi)) \Big( \frac{\p \Phi}{\p \xi_i}(s;\xi)-P(s;\xi)\cdot \frac{\p Q}{\p \xi_i}(s;\xi) \Big)\\
&&\quad = -\frac{\p\tilde{H}}{\p \Phi}(Q(s;\xi),P(s;\xi),\Phi(s;\xi)) \Big( \frac{\p \Phi}{\p \xi_i}(s;\xi)-P(s;\xi)\cdot \frac{\p Q}{\p \xi_i}(s;\xi) \Big),\\
&& \frac{\p \Phi}{\p \xi_i}(s;\xi)-P(s;\xi)\cdot \frac{\p Q}{\p \xi_i}(s;\xi)\equiv 0,\quad \forall\,\xi,\,\, \forall\,s.
\end{eqnarray*}
Hence, we see that
$$ D_{(s;\xi)}\Phi(s;\xi)= P(s;\xi) D_{(s;\xi)}Q(s;\xi)= P(s;\xi) D_{(s;\xi)} F(s,\xi).$$
Now we define
$$u(t,x):=\Phi(F^{-1}(t,x))=\Phi(t;\varphi(t,x)):O_2\to \R,$$
which is apparently $C^1$-smooth. Observe that
\begin{eqnarray*}
(\p_t u(t,x),\nabla u(t,x))&=&D_{(t,x)}u(t,x)=D_{(s;\xi)}\Phi(F^{-1}(t,x))D_{(t,x)}F^{-1}(t,x)\\
&=&P(F^{-1}(t,x)) D_{(s;\xi)} F(F^{-1}(t,x))D_{(t,x)}F^{-1}(t,x)\\&=&P(F^{-1}(t,x))
=P(t;\varphi(t,x)).
\end{eqnarray*}
Therefore, with \eqref{A10} and \eqref{A13}, we see that $u$ is in fact $C^2$-smooth and satisfies
\begin{eqnarray*}
0&=&\tilde{H}(Q(s;\xi), P(s;\xi);\Phi(s;\xi))|_{(s,\xi)=F^{-1}(t,x)=(t,\varphi(t,x))}\\
&=&\tilde{H}(t,x,\p_t u(t,x),\nabla u(t,x),u(t,x))\\
&=&\p_t u(t,x)+H(t,x,\nabla u(t,x),u(t,x)),\,\,\,\forall\,(t,x)\in O_2,\\
u(0,x)&=&\Phi(0;x)=w(x),\,\,\,\forall\,x\in O_2|_{t=0}.
\end{eqnarray*}
The uniqueness of \eqref{gHJ} follows from the uniqueness of the characteristic ODEs.

As a conclusion, {\it all one needs to do to solve  \eqref{gHJ} by the method of characteristics are
\begin{itemize}
\item Analyze \eqref{CH1}, \eqref{CH2} and \eqref{CH3} with $t(s)=s$, where  \eqref{CH4} and \eqref{CH5} are not explicitly necessary;
\item For the map $F(s,\xi):=(s,x(s,\xi))$, find an open set $O_1\subset\R\times\R^N$ with $O_1|_{t=0}\neq \emptyset$ such that $F:O_1\to O_2=F(O_1)$  is a $C^1$-diffeomorphism;
\item Define $u(t,x):=\Phi(t;\varphi(t,x)):O_2\to\R$ with $F^{-1}(t,x)=(t,\varphi(t,x))$.
\end{itemize}
}
\appendix
\def\thesection{Appendix 2: ODEs on closed sets and flow invariance}
\section{}

Let $J \subset \R$ be an open interval, $V \subset \R^N$ open and $g: J \times V \to \R^N$ continuous.
By Peano's theorem, the initial value problem (IVP for short)
\begin{equation}\label{eq:1}
x'(t)= g\big(t, x(t)\big) \text{ on } J, \quad  x(t_0)=x_0
\end{equation}
has a local (classical) solution for every $(t_0, x_0) \in J \times V$, i.e.\ there is $\ve=\ve(t_0, x_0) >0$ and a $C^1$-function
$x: I_\ve \to \R^N$, with $I_\ve=(t_0 -\ve, t_0 + \ve)$, such that \eqref{eq:1} is satisfied in every point.

In this situation, a closed set $K \subset V$ is said to be positive (negative) flow invariant for \eqref{eq:1}, if every solution of \eqref{eq:1} that starts in $t=t_0$ at a point $x_0 \in K$ stays inside $K$ for all (admissible) $t > t_0$ ($t < t_0$). In this case, one also speaks of forward (backward) invariance of $K$ for the right-hand side $g$. A closed set $K$ is called flow invariant, or just invariant, for \eqref{eq:1}, if $K$ is both positive and negative flow invariant for \eqref{eq:1}.

Since classical solutions for \eqref{eq:1} with (only) continuous right-hand side need not be unique, it might happen that, for given closed $K \subset V$, one solution starting in $x_0$ stays in $K$, while another solution leaves $K$.
The autonomous standard example for non-uniqueness, namely $g(x)=2\sqrt{|x|}$ on $V=\R$, already shows this behavior with $K:=\{0\}$ and $x_0=0$. One therefore calls a closed $K \subset V$ weakly (positively or negatively) flow invariant for \eqref{eq:1}, if $t_0 \in J$ and $x_0 \in K$ implies the existence of one  solution staying in $K$ (for $t>t_0$ or $t<t_0$, respectively).

We are interested in situations of unique solvability of \eqref{eq:1}, which holds true if $g$ is jointly continuous and locally Lipschitz continuous in $x$, i.e.\ for every $(t_0, x_0) \in J \times V$ there is $\delta= \delta(t_0, x_0) >0$ and $L=L(t_0, x_0)>0$ such that
\begin{equation}\label{eq:2}
|g(t,x)-g(t,\bar{x})| \leq L |x-\bar{x}| \text{ for all } t \in I_\delta \cap J, \; x,\bar{x} \in B_\delta (x_0)\cap V.
\end{equation}
In this case, the Picard-Lindel\"of theorem yields unique solvability of \eqref{eq:1}. Therefore, weak flow invariance then is the same as flow invariance.
Let us note in passing that forward (backward) unique solvability of \eqref{eq:1} for continuous $g$ holds under weaker additional assumptions such as one-sided Lipschitz continuity in $x$. More precisely, if for every $(t_0, x_0) \in J \times V$ there is $\delta=\delta(t_0, x_0) >0$ and a $k=k_{t_0,x_0} \in L^1(J)$ such that
\begin{equation}\label{eq:3}
\left\langle g(t,x) -g(t,\bar{x}), x-\bar{x}\right\rangle \leq k(t) |x-\bar{x}|^2 \text{ for } t \in J, \; x, \bar{x} \in B_\delta(x_0) \cap V,
\end{equation}
then forward uniqueness holds for \eqref{eq:1}.

Evidently, flow invariance of $K$ requires conditions on $g$. The extreme case is $K=\{x_0\}$, where $g(t, x_0)=0$ for $t \in I_\ve$ is required. In the general case, if $x_0 \in K$ and $x(t) \in K$ on $[t_0, t_0 +\ve)$ holds, then
\begin{equation*}\label{eq:4}
x(t_0 +h)= x_0 +h g(t_0, x_0) + e (h; t_0, x_0) \in K \text{ for } 0 \leq h < \ve
\end{equation*}
with a remainder term $e(\cdot; t_0, x_0)$, which satisfies
\begin{equation*}\label{eq:5}
\lim\limits_{h \to 0+} \frac{1}{h} e (h; t_0, x_0)=0.
\end{equation*}
Hence
\begin{equation*}\label{eq:6}
{\rm dist}\,\big(x_0+h g(t_0, x_0), K\big) \leq \left|e (h; t_0, x_0)\right|= o(h) \text{ as } h \to 0+.
\end{equation*}
Consequently, the condition
\begin{equation*}\label{eq:7}
g(t_0, x_0) \in \widetilde{T}_K(x_0)
\end{equation*}
with the so-called ``tangent cone''
\begin{equation*}\label{eq:8}
\widetilde{T}_K(x)=\left\{z \in X: \lim\limits_{h \to 0+} h^{-1} {\rm dist}\,(x+h z, K)=0\right\} \text{ for } x \in K
\end{equation*}
is a necessary condition, where $X=\R^N$.
Actually, it turns out that the apparently weaker condition
\begin{equation}\label{eq:9}
g(t,x) \in T_K(x) \quad\text{ for } t \in J, \; x \in K
\end{equation}
with the so-called Bouligand contingent cone
\begin{equation*}
T_K(x)=\left\{z \in X: \liminf\limits_{h \to 0+} h^{-1} {\rm dist}\, (x+ hz,K)=0\right\} \text{ for } x \in K
\end{equation*}
is sufficient for positive flow invariance of the closed set $K \subset V$; as above, $X=\R^N$ here.
This is a direct consequence of the following result on existence of solutions for ordinary differential equations (ODEs) on closed sets.
\medskip

\noindent{\bf Theorem A1.}
{\it Let $J=(a,b) \subset \R$, $K\subset \R^N$ closed and $g:J \times K \to \R^N$ continuous.
Then, given any $(t_0, x_0) \in J \times K$, the IVP \eqref{eq:1} has a local (classical) forward solution if and only if $g$ satisfies the subtangential condition \eqref{eq:9}. }
\medskip

Proofs of Theorem A1 can be found in \cite{AmannODEs}, \cite{DeLN} or \cite{Jan-Wilke-ODEs}.
We call elements from $T_K(x)$ subtangential vectors (to $K$ at $x$). Observe that in Theorem A1 the right-hand side $g(t,x)$ is only defined for $x \in K$, hence the constraint $x(t) \in K$ is incorporated into the domain of definition of $g$.
If $g$ is defined on the larger domain $J \times V$ with $V \supset K$, then Theorem A1 yields weak positive flow invariance of $K$ as only continuity of $g$ is assumed.
A solution which stays inside $K$ is then also called viable \cite{AuViab}.
Positive flow invariance follows if forward uniqueness for \eqref{eq:1} holds. Since backward solutions of \eqref{eq:1} are equivalent to forward solutions for $\widetilde{g} (t,x):= -g(2t_0-t,x)$, we also obtain a characterisation of negative flow invariance. Together, this is contained in\medskip

\noindent{\bf Theorem A2.}
{\it Let $J=(a,b) \subset \R$, $K \subset \R^N$ closed and $g: J \times K \to \R^N$ continuous.
Then the following statements hold true:
\begin{enumerate}
	\item[(a)]
	The IVP  \eqref{eq:1} admits local solutions on some $I_\ve$ (i.e.\ forward \& backward) iff
		\[
		g(t,x) \in T_K(x) \text{ and } -g(t,x) \in T_K(x) \text{ on } J \times K.\vspace{-0.1in}
	\]
	\item[(b)]
	If $g$ also satisfies the one-sided Lipschitz condition \eqref{eq:3} and satisfies \eqref{eq:9}, then IVP \eqref{eq:1} has a unique local forward solution for every $t_0 \in J$ and $x_0 \in K$.
	\item[(c)]
	If $g$ is locally Lipschitz continuous in $x$ (in the sense of \eqref{eq:2} above), with $\pm g(t,x) \in T_K(x)$ for all $t \in J$, $x \in K$, then IVP \eqref{eq:1} has a unique local solution on some $I_\ve$ (i.e.,\ forward \& backward) for every $t_0 \in J$ and $x_0 \in K$.
	
	If $g$ is defined on $J \times V$ with some open $V \supset K$, then $K$ is positive flow invariant for \eqref{eq:1} in the situation of (b), while $K$ is flow invariant for \eqref{eq:1} in the situation as described in (c).
\end{enumerate}
}

Theorems A1 and A2 have been generalized in several directions. See \cite{DeLN} for extensions to ODEs in Banach space, \cite{MDE} for extensions to differential inclusions (in $\R^N$ as well as in real Banach spaces), \cite{Bo2} for time-dependent constraints, i.e.\ IVP \eqref{eq:1} with the additional condition that $x(t) \in K(t)$ is to hold and \cite{Bo9}, \cite{Bo_Cara} for extensions to more general evolution problems.
Due to Theorem A1 and A2, information on flow invariance of given sets can be -- to a large extent -- obtained from properties of the set of subtangential vectors and their calculation.


\end{document}